\documentclass[final,12pt]{amsart}
\usepackage{amssymb, amsmath, amsthm,geometry}

\usepackage{mathrsfs}  
\usepackage{enumerate}
\usepackage{cite}

\usepackage{xcolor}
\usepackage{tikz}
\usetikzlibrary{arrows,shapes,chains}
\usetikzlibrary{arrows.meta}

\usepackage[all,cmtip, color]{xy}

\numberwithin{equation}{section}

\usepackage{graphicx}
\usepackage[colorlinks=true, citecolor=blue, linkcolor=blue, urlcolor=blue]{hyperref}
\newgeometry{asymmetric, centering}
 \usepackage{xcolor}
 
\usepackage{tikz}
\usetikzlibrary{shapes.geometric}
\usepackage{ifdraft}
\ifoptionfinal{
\usepackage[disable]{todonotes}
}{
\usepackage[norefs, nocites]{refcheck}
\usepackage{graphicx}
\usepackage{epstopdf}

\usepackage[notref, notcite]{showkeys}
\usepackage[bordercolor=white, color=white]{todonotes}
}

\makeatletter\providecommand\@dotsep{5}\def\listtodoname{List of Todos}\def\listoftodos{\hypersetup{linkcolor=black}\@starttoc{tdo}\listtodoname\hypersetup{linkcolor=blue}}\makeatother

\newtheorem{lemma}{Lemma}[section]
\newtheorem{proposition}{Proposition}[section]
\newtheorem{theorem}{Theorem}[section]

\newtheorem{definition}{Definition}[section]

\newtheorem{issue}{Issue}

\theoremstyle{remark}

\newcommand{\bel}{\begin{equation} \label}
\newcommand{\ee}{\end{equation}}
\def\beq{\begin{equation}}
\def\eeq{\end{equation}}
\newcommand{\bea}{\begin{eqnarray}}
\newcommand{\eea}{\end{eqnarray}}
\newcommand{\beas}{\begin{eqnarray*}}
\newcommand{\eeas}{\end{eqnarray*}}

\renewcommand{\div}{\mathrm{div}\,}  

\def\A{\mathcal A}

\def\C{\mathcal C}

\def\R{\mathbb R}

\def\P{\mathcal P}

\def\M{\mathbb M}

\def\WF{\mathrm{WF}}

\def\out{\mathrm{out}}

\def\b{\backslash}
\def \into{\mathrm{in}}

\def \rem{\mathrm{rem}}

\def \id{\mathrm{id}}

\def \fre{\mathrm{fre}}

\def\l{\left}
\def\r{\right}

\def\D{\mathbb D}

\renewcommand{\leq}{\leqslant}
\renewcommand{\geq}{\geqslant}
\def \e{\varepsilon}

\def \c{\boldsymbol}

\def\p{\partial}
\newcommand\rotom{\mho}

\date{}
\title
[Stable inversion of potential]{Stable inversion of potential in nonlinear wave equations with cubic nonlinearity}

\author[X. Chen]{Xi Chen}
\address{Shanghai Center for Mathematical Sciences, Fudan University, Shanghai 200438, China. }\email{xi\_chen@fudan.edu.cn}
\author[S. Lu]{Shuai Lu}
\address{School of Mathematical Sciences, Fudan University, Shanghai 200433, China. }
\email{slu@fudan.edu.cn}
\author[R. Zhang]{Ruochong Zhang}
\address{School of Mathematical Sciences, Fudan University, Shanghai 200433, China. }
\email{22110180055@m.fudan.edu.cn}

\begin{document}

\maketitle

\begin{abstract} This paper investigates inverse potential problems of wave equations with cubic nonlinearity. We develop a methodology for establishing stability estimates for inversion of lower order coefficients. The new ingredients of our approach include trilinear approximations of nonlinear response operators, symbol estimates of distorted plane waves, and lower order symbol calculus. 
\end{abstract}

\tableofcontents

\section{Introduction}\label{sec: introduction}

Consider Minkowski spacetime $\M$ with signature $(3, 1)$ and D'alembertian operator $\Box$. We adopt Cartesian coordinates $x = (x_0, x_1, x_2, x_3) = (t, x')$ such that
\begin{align*}
	\M &= \left\{(x_0, x_1, x_2, x_3) \in \R^4\right\} =  \left\{(t, x') \in \R \times \R^3\right\}; \\ 
	 \Box &= \partial_{x_0}^2 - \sum_{j=1}^3 \partial_{x_j}^2  = \partial_t^2 - \partial_{x'}^2.\end{align*}
Suppose  $V \in C^\infty(\M)$ is a real potential and $h \in C^\infty(\M)$ is a non-vanishing variable coefficient on $\M$. We consider a sourced semilinear wave equation \begin{equation}\label{eqn : semilinear model}
\left\{\begin{aligned}\Box u + Vu + h u^3 &= f && t > 0;\\ u &\equiv 0 && t \leq 0.\end{aligned}\right.
\end{equation}  Then \eqref{eqn : semilinear model}   admits a unique solution in finite times as long as the source $f$ is sufficiently small and regular.

In the present paper, we shall investigate inverse potential problems of \eqref{eqn : semilinear model} in the following model. 

 Denote by $\mathcal{O}$ the ball centred at $0 \in \R^3$ with radius $\rho > 0$. For $T > 0$, we set a cylinder $\mho$  
\begin{align*} 
	\mho &:=  (0, T) \times \mathcal{O} = \left\{ (t, x') \in (0, T) \times \R^3 : |x'| < \rho  \right\}  \end{align*}
and a causal diamond   $\D$
\begin{align*}  
	\D &:= \left\{ (t, x') \in (0, T) \times \R^3 : |x'| \leq t+ \rho, |x'| \leq \rho + T -t  \right\}.\end{align*}
Moreover, we take $\mathcal{O}_i, i = 1, 2$ to be a ball centred at $0 \in \R^3$ with radius $\rho_i > 0$ such that $\rho > \rho_1 > \rho_2$. For small $t_2>t_1>0$,  denote smaller cylinders and diamonds
\begin{align*} 
	\rotom_i &:= (t_i,T-t_i)\times \mathcal{O}_i = \left\{ (t, x') \in (t_i, T - t_i) \times \R^3 : |x'| < \rho_i  \right\}  ,\\
	\mathbb{D}_i &:= \left\{ (t, x') \in (t_i,T-t_i) \times \R^3 : |x'| \leq t-t_i+ \rho_i  , |x'| \leq \rho_i + T-t_i -t  \right\}  .
\end{align*} 
See Figure \ref{fig:figure1} and Figure \ref{fig:figure2}.

\tikzset{every picture/.style={line width=0.75pt}} 

\begin{figure}[htbp]
	\centering
	\begin{minipage}{0.45\textwidth}
		\centering
		\begin{tikzpicture}[x=0.75pt,y=0.75pt,yscale=-1,scale=0.8]
			
			\draw   (180.21,43.81) -- (180.21,269.35) .. controls (180.21,275.98) and (162.31,281.35) .. (140.22,281.35) .. controls (118.14,281.35) and (100.23,275.98) .. (100.23,269.35) -- (100.23,43.81) .. controls (100.23,37.18) and (118.14,31.81) .. (140.22,31.81) .. controls (162.31,31.81) and (180.21,37.18) .. (180.21,43.81) .. controls (180.21,50.44) and (162.31,55.81) .. (140.22,55.81) .. controls (118.14,55.81) and (100.23,50.44) .. (100.23,43.81) ;
			\draw  [color={rgb, 255:red, 15; green, 2; blue, 208 }  ,draw opacity=1 ] (168.44,70.27) -- (168.44,248.68) .. controls (168.44,253.35) and (155.81,257.14) .. (140.22,257.14) .. controls (124.63,257.14) and (112,253.35) .. (112,248.68) -- (112,70.27) .. controls (112,65.6) and (124.63,61.81) .. (140.22,61.81) .. controls (155.81,61.81) and (168.44,65.6) .. (168.44,70.27) .. controls (168.44,74.95) and (155.81,78.74) .. (140.22,78.74) .. controls (124.63,78.74) and (112,74.95) .. (112,70.27) ;
			\draw  [color={rgb, 255:red, 245; green, 28; blue, 8 }  ,draw opacity=1 ] (162.22,90.22) -- (162.22,231.52) .. controls (162.22,235.1) and (152.55,238) .. (140.61,238) .. controls (128.68,238) and (119,235.1) .. (119,231.52) -- (119,90.22) .. controls (119,86.64) and (128.68,83.74) .. (140.61,83.74) .. controls (152.55,83.74) and (162.22,86.64) .. (162.22,90.22) .. controls (162.22,93.8) and (152.55,96.7) .. (140.61,96.7) .. controls (128.68,96.7) and (119,93.8) .. (119,90.22) ;
			\draw  [dash pattern={on 4.5pt off 4.5pt}] (100.23,269.35) .. controls (126.89,259.74) and (153.55,259.74) .. (180.21,269.35) ;
			\draw  [color={rgb, 255:red, 15; green, 2; blue, 208 }  ,draw opacity=1 ][dash pattern={on 4.5pt off 4.5pt}] (112,248.68) .. controls (130.81,241.11) and (149.63,241.11) .. (168.44,248.68) ;
			\draw  [color={rgb, 255:red, 240; green, 9; blue, 9 }  ,draw opacity=1 ][dash pattern={on 4.5pt off 4.5pt}] (119,231.52) .. controls (133.41,224.37) and (147.81,224.37) .. (162.22,231.52) ;
		\end{tikzpicture}
		
		\caption{The outer, middle, and inner cylinders respectively represent $\rotom$, $\rotom_1$ and $\rotom_2$.}
		\label{fig:figure1}
	\end{minipage}
	\begin{minipage}{0.45\textwidth}
		\centering
		\begin{tikzpicture}[x=0.75pt,y=0.75pt,yscale=-1,scale=0.8]
			
			\draw   (530.21,29.66) -- (530.21,269) .. controls (530.21,275.62) and (512.31,280.99) .. (490.22,280.99) .. controls (468.14,280.99) and (450.23,275.62) .. (450.23,269) -- (450.23,29.66) .. controls (450.23,23.04) and (468.14,17.67) .. (490.22,17.67) .. controls (512.31,17.67) and (530.21,23.04) .. (530.21,29.66) .. controls (530.21,36.29) and (512.31,41.66) .. (490.22,41.66) .. controls (468.14,41.66) and (450.23,36.29) .. (450.23,29.66) ;
			\draw  [color={rgb, 255:red, 15; green, 2; blue, 208 }  ,draw opacity=1 ] (520.57,59.78) -- (520.57,239.07) .. controls (520.57,244.1) and (506.97,248.18) .. (490.2,248.18) .. controls (473.43,248.18) and (459.83,244.1) .. (459.83,239.07) -- (459.83,59.78) .. controls (459.83,54.75) and (473.43,50.67) .. (490.2,50.67) .. controls (506.97,50.67) and (520.57,54.75) .. (520.57,59.78) .. controls (520.57,64.81) and (506.97,68.89) .. (490.2,68.89) .. controls (473.43,68.89) and (459.83,64.81) .. (459.83,59.78) ;
			\draw  [color={rgb, 255:red, 245; green, 28; blue, 8 }  ,draw opacity=1 ] (510.17,80.68) -- (510.17,219.66) .. controls (510.17,222.98) and (501.2,225.67) .. (490.13,225.67) .. controls (479.06,225.67) and (470.09,222.98) .. (470.09,219.66) -- (470.09,80.68) .. controls (470.09,77.36) and (479.06,74.67) .. (490.13,74.67) .. controls (501.2,74.67) and (510.17,77.36) .. (510.17,80.68) .. controls (510.17,84) and (501.2,86.69) .. (490.13,86.69) .. controls (479.06,86.69) and (470.09,84) .. (470.09,80.68) ;
			\draw  [dash pattern={on 4.5pt off 4.5pt}] (450.23,269) .. controls (476.89,259.39) and (503.55,259.39) .. (530.21,269) ;
			\draw  [color={rgb, 255:red, 15; green, 2; blue, 208 }  ,draw opacity=1 ][dash pattern={on 4.5pt off 4.5pt}] (462.39,239.33) .. controls (481.2,231.76) and (500.02,231.76) .. (518.83,239.33) ;
			\draw  [color={rgb, 255:red, 240; green, 9; blue, 9 }  ,draw opacity=1 ][dash pattern={on 4.5pt off 4.5pt}] (470.09,219.66) .. controls (483.76,213.29) and (497.43,213.29) .. (511.09,219.66) ;
			\draw    (329.5,150.67) -- (450.23,269) ;
			\draw    (329.5,150.67) -- (450.23,29.66) ;
			\draw [color={rgb, 255:red, 15; green, 2; blue, 208 }  ,draw opacity=1 ]   (370.01,150.43) -- (459.83,239.07) ;
			\draw [color={rgb, 255:red, 15; green, 2; blue, 208 }  ,draw opacity=1 ]   (370.01,150.43) -- (459.83,59.78) ;
			\draw [color={rgb, 255:red, 15; green, 2; blue, 208 }  ,draw opacity=1 ]   (520.57,59.78) -- (609.98,149.43) ;
			\draw [color={rgb, 255:red, 15; green, 2; blue, 208 }  ,draw opacity=1 ]   (520.16,240.08) -- (609.98,149.43) ;
			\draw    (530.21,29.66) -- (650.04,150.28) ;
			\draw    (530.21,269) -- (650.04,150.28) ;
			\draw   (329.5,150.67) .. controls (436.35,190.47) and (543.19,190.47) .. (650.04,150.67) ;
			\draw  [dash pattern={on 4.5pt off 4.5pt}] (329.5,150.67) .. controls (436.35,111.17) and (543.19,111.17) .. (650.04,150.67) ;
			\draw [color={rgb, 255:red, 245; green, 28; blue, 8 }  ,draw opacity=1 ]   (400.17,150.67) -- (470.09,219.66) ;
			\draw [color={rgb, 255:red, 245; green, 28; blue, 8 }  ,draw opacity=1 ]   (400.17,150.67) -- (470.09,80.68) ;
			\draw [color={rgb, 255:red, 245; green, 28; blue, 8 }  ,draw opacity=1 ]   (510.17,80.68) -- (580.09,149.67) ;
			\draw [color={rgb, 255:red, 245; green, 28; blue, 8 }  ,draw opacity=1 ]   (511.09,219.66) -- (581.02,149.67) ;
			\draw  [color={rgb, 255:red, 15; green, 2; blue, 208 }  ,draw opacity=1 ] (370.01,149.43) .. controls (450,176.36) and (529.99,176.36) .. (609.98,149.43) ;
			\draw  [color={rgb, 255:red, 15; green, 2; blue, 208 }  ,draw opacity=1 ][dash pattern={on 4.5pt off 4.5pt}] (370.01,149.43) .. controls (450,123.97) and (529.99,123.97) .. (609.98,149.43) ;
			\draw  [color={rgb, 255:red, 245; green, 28; blue, 8 }  ,draw opacity=1 ] (400.17,150.67) .. controls (460.14,163.23) and (520.12,163.23) .. (580.09,150.67) ;
			\draw  [color={rgb, 255:red, 245; green, 28; blue, 8 }  ,draw opacity=1 ][dash pattern={on 4.5pt off 4.5pt}] (400.17,149.67) .. controls (460.14,137.14) and (520.12,137.14) .. (580.09,149.67) ;
		\end{tikzpicture}
		\caption{The outer, middle, and inner diamonds respectively represent $\D$, $\D_1$ and $\D_2$.}
		\label{fig:figure2}
	\end{minipage}
\end{figure}


Assume that we can do active local measurements $\P_{V, h}$ within $\mho$ for \eqref{eqn : semilinear model}. Specifically, we make sources  in a sufficiently small neighbourhood 
$$\mathscr{B}^s_\varrho(\mho):= \left\{ f \in \mathscr{H}^s(\mho) : \|f\|_{\mathscr{H}^s(\mho)} \leq \varrho  \right\}$$
 of $0$ in the space   \begin{equation}\label{Domian of source-to-solution map}
	\mathscr{H}^{s}\left(  \mho \right) := \bigcap_{0 \leq k \leq s} C^k\l((0,T); H^{s - k} (\mathcal{O}) \r),
\end{equation} where $\varrho > 0$ is a sufficiently small constant and $s \geq 4$. The routine local existence  theory of semilinear hyperbolic PDEs (see e.g. \cite[Theorem 6.3.1]{Rauch-book}) guarantees that there exists a unique solution $u$ to \eqref{eqn : semilinear model} lying in the space  $$
\mathscr{H}^{s}\left(  \M_{T}  \right) :=  \bigcap_{0 \leq k \leq s} C^k\l((0,T); H^{s - k} (\R^3) \r),
$$ with $\M_{T} : = (0, T) \times \R^3$.
Consequently, the following source-to-solution map
\[\begin{aligned}
	\P_{V, h} : f &\longmapsto u |_\mho, && \mbox{for $f \in \mathscr{B}^s_\varrho(\mho)$ and $(f, u)$ solves \eqref{eqn : semilinear model}},  
\end{aligned}\] 
is well-defined.

Given the information of $\P_{V, h}$ in $\mho$, we can reconstruct the nonlinearity and potential in $\D\setminus \mho$ stably in the H\"older sense. More precisely,  we shall prove
\begin{theorem}\label{theorem: main theorem}
	Suppose  $(V,h)$ and $(\tilde{V},\tilde{h})$ are two pairs of smooth potential and non-vanishing coefficient functions in \eqref{eqn : semilinear model}, which agree in $\mho$.  Their source-to-solution maps are defined accordingly as above. 
	If there exists some $\delta \in (0, \varrho)$ such that  \begin{equation}
		\label{error of the source-to-solution map V h}
		\l\| \P_{V,h}(f) - \P_{\tilde{V},\tilde{h}}(f)\r\|_{\mathscr{H}^s(\rotom)} \le \delta, \quad  \forall f\in \mathscr{B}^s_\varrho(\mho),
	\end{equation}	
	then there hold the following stability estimates of H\"older type for  simultaneous inversion of potential  and nonlinearity,  
	\begin{align}
		\label{eqn : stability h}\l\| h-\tilde{h}\r\|_{L^{\infty}(\overline{\D_1})} \lesssim \delta^{\frac{2}{15(s+2)}},
	\end{align}
	and
	\begin{align}
		\label{eqn : stability V} \l\| V-\tilde{V}\r\|_{L^{\infty}(\overline{\D_2})} \lesssim \l(\l(\delta^{\frac{2}{5}} + \l\|h-\tilde{h} \r\|_{L^{\infty}(\overline{\D_1})}\r)^{\frac{2}{(3s+7)(6s+5)}} + \l\|h-\tilde{h} \r\|_{L^{\infty}(\overline{\D_1})}\r)^{\frac{1}{2}}.
	\end{align}

\end{theorem}

Unique inversion of coefficient for hyperbolic equations has been studied extensively. For time-independent cases,  Belishev--Kurylev's boundary control method \cite{BK} together with Tataru's unique continuation approach \cite{T} recovers coefficients for generic  hyperbolic equations. Nevertheless, this set of techniques in general break down for time-dependent wave equations as Alinhac's example \cite{A} suggests. Kurylev--Lassas--Uhlmann \cite{KLU} pioneered the higher order linearization scheme to reconstruct principal coefficients for nonlinear time-dependent hyperbolic PDEs. The core of this methodology is to exploit the nonlinear scattering data of conormal waves and approximate nonlinear waves by linear conormal waves.   Then principal coefficients can be reconstructed from the travel times of linearized conormal waves. 

Compared with principal terms, the information of lower order terms, being less visible from the viewpoint of singularities, is encoded in distributional and symbolic calculus of lower order.   Based on the higher order linearization, detection of lower order coefficients for nonlinear hyperbolic equations also has been considered. Chen--Lassas--Oksanen--Paternain \cite{CLOP} introduced the method of broken light-ray to determine subprincipal coefficients for cubic semilinear wave equations. Unique inversion of potential in the similar setting  was proved by Feizmohammadi--Oksanen \cite{FO}.

We refer the readers to \cite{SaBarreto-Stefanov,FLO,LUW, Hintz-Uhlmann-Zhai-IMRN,Hintz-Uhlmann-Zhai-CPDE,Kian, SaBarreto-Uhlmann-Wang, WZ} for more recent advances in uniqueness of inversion for nonlinear hyperbolic PDEs, and to \cite{KLOU,Uhlmann-Wang-CPAM, CLOP-CMP,CLOP-Annalen, deHoop-Uhlmann-Wang-2019, deHoop-Uhlmann-Wang-2020, Uhlmann-Zhai-Annalen,Uhlmann-Zhai-JMPA} for applications in physical models.  

Stability of inversion is more involved and entails subtle quantitative estimates for local measurements. The framework of boundary control is essentially qualitative and thus unlikely to deliver desired quantitative results. On simple Riemannian manifolds, Stefanov--Uhlmann \cite{SU-IMRN} reduced the stable determination of coefficients for linear wave equations to that of metric rigidity \cite{SU-JAMS} and X-ray transforms \cite{SU-DUKE}. In the non-simple case, Bao--Zhang \cite{BZ} established sensitivity results on manifolds with fold-regular conditions by utilizing Gaussian beams. 

The research of stable inversion for nonlinear hyperbolic equations is rather recent. Lassas--Liimatainen--Potenciano-Machado--Tyni \cite{LLPT22, LLPT24, LLPT-APDE} established stability results for semilinear wave equations in bounded domains. They proved that the coefficient of nonlinear terms can be stably recovered from full boundary measurements as long as the observation errors of Dirichlet-to-Neumann maps are small.   

We aim to develop a general framework of stable inversion of lower order coefficients for nonlinear hyperbolic equations with a small amount of datum. Compared with stability for linear equations, uniqueness of recovery, and inversion in full data DtN models, we encounter the following novel hurdles :
\begin{issue}
  How to measure nonlinear response operators? \end{issue} Unlike linear equations,  there is no  natural operator space to accommodate nonlinear response operators. However, it is important in stability problems to convert the error estimates for nonlinear response operators to certain operator norms. 
For this purpose, we  {introduce some trilinear response operators to approximate nonlinear source-to-solution maps} for \eqref{eqn : semilinear model}. The difference of two such trilinear operators turns out to be the leading part of the difference of two source-to-solution maps. Therefore, it is legitimate to  {employ the norms of trilinear operators to measure nonlinear response operators}.
	
	\begin{issue} How to establish quantitative estimates for underlying coefficients? \end{issue} For the purpose of uniqueness, it suffices to prove injectivity results for some transformations. For hyperbolic PDEs, this is usually achieved by doing singularity analysis and removing non-singular and less singular terms. With regard to stabilities, the key matter is that those neglected terms in uniqueness problems, such as smooth terms, will no longer be negligible but become unavoidable errors. In order to quantify to what extent such errors affect inversion results, it has to  {entail rather precise quantitative estimates for relevant transformations}. The trilinear response operators are essentially trilinear Fourier integral operators and their symbols actually include some integral transformations of underlying coefficients. Hence some refined symbol calculi and estimates are pivotal to derive quantitative estimates for those transformations. 
	
	\begin{issue} How to capture quantitative data of potentials in nonlinear equations? \end{issue} Zero-order coefficients are less visible than others, as potentials never appear in the principal terms of our local measurements. Consequently, routine principal symbol calculus provides no information for them. To remedy this, we have to go further and  {analyse transport equations for lower order symbols to detect integral transformations of potentials}.  In additional, the product in nonlinearity severely complicates the symbolic structure of nonlinear conormal waves. Therefore, we  {perform
	  subtle order analysis  to find out the exact leading terms and subleading terms in the expansion of multilinear conormal waves}. 
 
\begin{issue}
	How to achieve stable inversion with  a small amount of datum?
\end{issue}

Compared with classical DtN models, the source-to-solution type measurements made in Theorem \ref{theorem: main theorem} carry rather limited data. DtN models utilize the boundary information to recover coefficients in the interior, while Theorem \ref{theorem: main theorem} requires only the knowledge in a small spatial neighbourhood to  reconstruct the coefficients in the exterior. Such amount of information is proven to be insufficient for inverse coefficient problems of time-dependent linear equations, see e.g. \cite{Stefanov-Yang-APDE}. For nonlinear equations, one may however exploit the nonlinear scattering data to prove the uniqueness pointwisely. On the other hand, the measurement for stability problems, which is in the form of Sobolev norms of solutions, is far less informative than the pointwise knowledge of solutions in uniqueness problems. Therefore, we    {employ more robust tools, such as energy estimates and stationary phase estimates,  to deal with such  measurement of neighbourhoodwise type}.

In summary, we propose the following roadmap to solve stable inversion problems of lower order coefficients in nonlinear wave equations. 


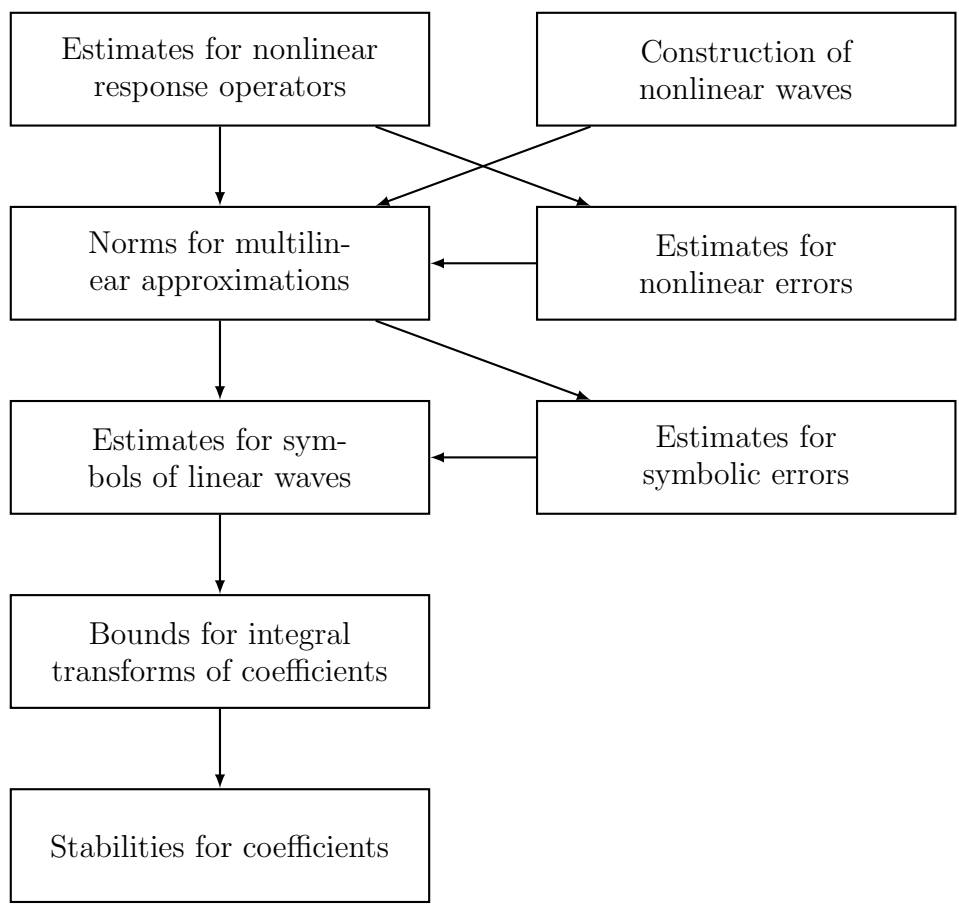
\begin{figure}[htbp]
    \tikzstyle{startstop} = [rectangle, rounded corners=0.5cm, thick, minimum width = 3cm, minimum height=1cm,text centered, draw = black, ]
    \tikzstyle{process} = [rectangle, thick, minimum width=5cm, minimum height=1.5cm, text centered, text width = 5cm, inner sep = 8pt, draw = black,]
    \tikzstyle{arrow} = [thick,->,>=latex]
    \centering
    \begin{tikzpicture}[node distance=3cm]
        \node (m1) [process] {Estimates for nonlinear response operators};
        \node (m2) [process, right = 40pt of m1] {Construction of nonlinear waves};
        \node (m3) [process, below = 30pt of m1] {Norms for multilinear approximations};
        \node (m4) [process, below = 30pt of m2] {Estimates for nonlinear errors};
        \node (m5) [process, below = 30pt of m3] {Estimates for symbols of linear waves};
        \node (m6) [process, below = 30pt of m4] {Estimates for symbolic errors};
           \node (m7) [process, below = 30pt of m5] {
           Bounds for integral transforms of coefficients};
           \node (m9) [process, below = 30pt of m7] {
          Stabilities for coefficients};
      \draw[arrow](m1) -- (m3);
      \draw[arrow](m3) -- (m5);
      \draw[arrow](m5) -- (m7);
      \draw[arrow](m7) -- (m9);
      \draw[arrow](m1) -- (m4);
      \draw[arrow](m2) -- (m3);
      \draw[arrow](m3) -- (m6);
      \draw[arrow](m4) -- (m3);
      \draw[arrow](m6) -- (m5);
    \end{tikzpicture}
\end{figure}

To focus on how to remedy the issues above, we look only at the cubic wave equation in Minkowski spacetime in the present paper.   In ongoing and future projects, we shall develop this framework for more general semilinear and quasilinear models in globally hyperbolic Lorentzian manifolds.

This paper is structured as follows. To begin with, we briefly review, in Section \ref{sec : linear waves}, the distributional and symbolic structures of conormal waves, a class of solutions useful for inversion. For the purpose of inverse potential problems, we, in particular, analyse the lower order symbols of conormal distributions in Section \ref{sec: lower order symbol calculus}.  Using such tools, we describe, in Section \ref{sec : nonlinear waves}, a class of special nonlinear conormal waves and their trilinear approximations. Subsequently, we conduct principal as well as lower order symbol calculi to quantify such trilinear approximations in Section \ref{sec : principal singularities}--\ref{sec : trilinear operators}. In the end, Theorem \ref{theorem: main theorem} will be proved, in Section \ref{sec : stability estimates}, via the trilinear approximations by employing quantitative estimates for lower order symbols and some truncated integrals.

\section{Linear distorted plane waves}
\label{sec : linear waves}

We commence from linear waves described by inhomogeneous wave equations
\begin{align}\label{eqn : linear wave with potential}\Box u + Vu = f.\end{align}

\subsection{Conormal distributions}

A useful class of special solutions to solve inverse problems are conormal distributions, also known as distorted plane waves.
Suppose a smooth $n$-manifold $X$ and the cotangent space $T^\ast_x X$ at $x \in X$ are furnished with local coordinates $(x^1, \cdots, x^n)$ and $(\xi_1, \cdots, \xi_n)$.
We say $Y$ is a $p$-submanifold of $X$ if $Y$ is equipped with local coordinates $x' := (x^1, \cdots, x^r)$ with $1 \leq r \leq n - 1$. The normal bundle $NY$ to $Y$ in $X$ is defined to be the quotient $N Y := T X / TY$, while the conormal bundle $N^\ast Y$ to $Y$ in $X$ is the dual bundle to $NY$. In local coordinates,  $N^\ast Y$ is locally expressed by $$\left\{\left(x', \xi''\right)\right\} := \left\{\left(x^1, \cdots, x^r, \xi_{r+1}, \cdots, \xi_n\right)\right\}.$$ With respect to the symplectic structure $(T^\ast X, dx^\alpha \wedge d\xi_\alpha)$, the conormal bundle $N^\ast Y$ is obviously a Lagrangian submanifold. 
\begin{definition}   The space of conormal distributions to $Y$ of order $m$, denoted by $I^{m}(N^\ast Y)$  for simplicity, consists of distributions,  which locally take the form 
	\[\int_{\R^{n-r}} e^{\imath x'' \cdot \xi''} a(x', \xi'') \,d\xi''\]  
	with amplitude $a \in S^{m +n/4 - (n-r)/2}(\R^r_{x'} \times \R^{n-r}_{\xi''})$.   For $A \in I^m(N^\ast Y)$, the principal symbol of $A$ is an invariantly defined function $\sigma[A]$ on $N^{\ast}Y$ such that
	\begin{align} \label{eqn: invariant definition of the principal symbol}
		\sigma[A](\alpha) &= e^{\imath\psi(\pi(\alpha),\alpha)} \langle A, \phi e^{-\imath\psi(x,\alpha)}\rangle \omega\in   S^{m+\frac{n}{4}}(N^\ast Y)/ S^{m+\frac{n}{4}-1}(N^\ast Y)  
	\end{align}
	where 
    \begin{itemize}
   \item $\phi \in C_0^\infty(X)$ is a half-density valued cut-off function supported near $\pi(\alpha)$;
   \item $\psi\in C^{\infty}(X \times N^\ast Y)$ is a  function such that $\psi$ is homogeneous of degree $1$ in $\alpha$ and the graph of $x\mapsto d_{x}\psi(x,\alpha)$ intersects $N^\ast Y$ transversally at $\alpha$;
   \item $\omega$ is the unity density of order $1/2$ on the fiber of $T^{\ast}X$ homogenous of degree $n/2$. 
   \end{itemize}
\end{definition}

Moreover, Melrose--Uhlmann \cite{MU} and Guillemin--Uhlmann \cite{GU81} formulated paired conormal distributions to address conormal waves around their sources. We next express them in local coordinates. Let $X$ be a smooth $n$-manifold with local coordinates $$x := \left(x', x'', x'''\right) \in \R^k \times \R^{n - k - d} \times \R^d$$ and $p$-submanifolds 
\begin{align*}Y_0 &:= \left\{x \in M : x' =0,  x'' = 0\right\}\\ 
	Y_1 &:= \left\{x \in M : x'' = 0\right\}.\end{align*} Consider two conormal bundles $N^\ast Y_0, N^\ast Y_1 \subset T^\ast X \setminus 0$ with local expression 
\begin{align*}N^\ast Y_0 &= \left\{ (x, \xi) \in T^\ast M : x' = 0, x'' = 0, \xi'''=0 \right\}   \\    
	N^\ast Y_1 &=  \left\{ (x, \xi) \in T^\ast M :  x'' = 0, \xi'=0, \xi'''=0 \right\},\end{align*} where $\xi := (\xi', \xi'', \xi''') \in T^\ast_xM$.
\begin{definition}\label{def: definition of paired conormal distribution}
	The space $I^{p, l}\left(N^\ast Y_0, N^\ast Y_1\right)$ consists of distributions of the form
	\[ \int_{\R^{n-d}} e^{\imath (x' \cdot \xi' + x'' \cdot \xi'')} b\left(  x''', \xi', \xi''  \right)  \, d\xi' d\xi''    \]
	where the smooth amplitude $b$ is an element of $$S^{M, M'}\left(\R^d_{x'''} \times \R^{n - k - d}_{\xi''} \times \R^k_{\xi'}\right), \quad \mbox{with $M = p - \frac{n}{4} + \frac{k}{2} + \frac{d}{2}$ and $M' = l - \frac{k}{2}$},$$ which means that  for any $x'''$ of a compact subset $K$ and multi-indices $\alpha, \beta, \gamma$,
	\[  \left|  \p_{x'''}^\alpha \p_{\xi'}^\beta \p_{\xi''}^\gamma b(x''', \xi', \xi'') \right| \lesssim_{K, \alpha, \beta, \gamma} \langle \xi', \xi'' \rangle^{M - |\beta|} \langle \xi'' \rangle^{M' - |\gamma|}.\]\end{definition}

For our purpose, it turns out to be enough to invoke the prototype of paired conormal distributions. That is the class $I^{p, -1/2}\left(N^\ast Y_0, N^\ast Y_1\right)$ with $k=1$ and $d=0$. In this case, we use the shorthand 
\[I^p\left(N^\ast Y_0, N^\ast Y_1\right) := I^{p, -1/2}\left(N^\ast Y_0, N^\ast Y_1\right).\]  

With paired conormal distributions, one can construct the solution to \eqref{eqn : linear wave with potential} with a point source. \begin{proposition}\label{prop : linear parametrix}
	Suppose $(\Lambda_0, \Lambda_1)$ is a transversally intersecting pair of conormal bundles of $T^\ast X$ such that \begin{itemize}
		\item $\Lambda_0 = N^\ast \{\c{x}\}$;
		\item $\Lambda_1$ is the bicharacteristic flow-out of $\Sigma(\Box) \cap   \mathrm{WF}(f)$ with $f \in I^{\mu + 3/2}(\Lambda_0)$. \end{itemize} Then the following holds.
	\begin{itemize} 
		\item The solution $u$ to \eqref{eqn : linear wave with potential} is a paired conormal distribution which lies in the class $I^{ \mu}(\Lambda_0, \Lambda_1)$. 
		\item The principal symbol   $\sigma[u]$ satisfies the transport equation \begin{equation}\label{symbol equation of principal symbol of one fold -linearization}
			\begin{aligned}
				\begin{cases}
					\mathscr{L}_{H_\Box}\sigma[u] = 0 \quad &\text{on }\Lambda_1\b\partial\Lambda_1,\\
					\sigma[u] = \mathscr{R}\l(\sigma\l[\Box\r]^{-1}\sigma[f]\r) &\text{on }\partial\Lambda_1
				\end{cases}
			\end{aligned}
		\end{equation} where \begin{itemize}
			\item $\sigma[u]$ is defined in \cite[Section 4]{MU};
			\item $H_\Box$ is the Hamiltonian vector field of  $\sigma[\Box]$ with respect to the symplectic structure of $T^\ast X$;
			\item $\mathscr{L}_\cdot$ denotes the Lie derivative along vector fields;
			\item  $\mathscr{R}$ is a symbol transition map from $\Lambda_0$ to $\p \Lambda_1$ introduced in \cite[(4.12)]{MU}. \end{itemize}\end{itemize} \end{proposition}
            \begin{proof}
             See \cite[Proposition 5.4, Proposition 6.6]{MU}. 
            \end{proof}


 \subsection{Principal symbol calculus}

Next, we shall go a bit further to compute the principal symbol and lower order symbols of a paired conormal distribution $u$ along a  bicharacteristic  curve $\beta(s) := (\gamma(s),\dot{\gamma}^{\flat}(s))$ emanating from $(\c{x},\c{\xi})\in \partial\Lambda_1$, where the superscript $\cdot^\flat : T_p X \rightarrow T_p^\ast X$ denotes the musical isomorphism. Without loss of generality, we translate $\c{x}$ to $0$ for convenience.

By \cite[(25.2.11)]{H4}, we Lie differentiate $\sigma[u] := \hat{u}\omega$ with non-vanishing smooth half-density  $\omega$ 
$$
\mathscr{L}_{H_\Box}(\hat{u}\omega) = \left(  H_\Box (\hat{u}) + \frac{1}{2}\l(\div H_\Box \r) \hat{u}\right) \omega.
$$ 
After plugging this into the transport equation \eqref{symbol equation of principal symbol of one fold -linearization}, we restrict   to the bicharacteristic curve $\beta(s)$ and divide the equation by $\omega$. It then gives
$$
H_\Box\hat{u}(\beta(s)) + \frac{1}{2}\l(\hat{u} \cdot \div H_\Box\r)(\beta(s)) = 0.
$$ 
By denoting the integral factor
\begin{equation}\label{integral factor of the transport equation and half-density}
	\rho(s):=\frac{1}{2}\int_{0}^{s}\div H_\Box(\beta(t))dt 
\end{equation} and using the fact
$$
H_\Box \hat{u}(\beta(s)) = \partial_s (\hat{u}\circ\beta)(s),
$$ 
we reduce the transport equation to
$$
\partial_{s}(\hat{u}\circ\beta)(s) + \hat{u}(\beta(s))\partial_{s}\rho(s) = 0.
$$
Multiplying  the equation by  $e^{\rho(s)}$ leads to
\begin{equation}\label{eqn: transport equation of the principal symbol of linear waves}
   \left\{\begin{aligned}
		\partial_{s}\l(e^{\rho(s)}\hat{u}\circ\beta(s) \r) &= 0 \\ 
		\hat{u}(\beta(0)) &= \hat{u}(0,\c{\xi}),\end{aligned}\right.
\end{equation}
which amounts to
\begin{equation}\label{eqn: solution of the transport equation without half-density}
e^{\rho(s)}\hat{u}(\beta(s)) \equiv \hat{u}(0;\c{\xi}).
\end{equation}
Write
\begin{equation}\label{eqn: factor alpha as a zero-th order half-density}
   \alpha(\beta(s)) := e^{-\rho(s)}\omega(\beta(s))\omega^{-1}(\beta(0)).
\end{equation}
By the proof of \cite[Proposition 1]{CLOP}, we see that $\alpha$ is strictly positive and homogeneous of degree $0$ in $\dot{\gamma}^{\flat}(s)$. Combining \eqref{eqn: solution of the transport equation without half-density} and \eqref{eqn: factor alpha as a zero-th order half-density} yields that 
\begin{equation}\label{eqn: principal symbol along bicharacteristics}
  \sigma[u](\beta(s)) = \alpha(\beta(s))\sigma[u](\beta(0)).
\end{equation}

If $\sigma[u]$ is positively homogeneous of degree $\mu + 1/2$, then
\begin{equation}\label{homogeneity of the principal symbol}
	e^{\rho(s)}\hat{u}(\gamma(s),   \lambda \dot{\gamma}^\flat(s)) = \lambda^\mu  e^{\rho(s)} \hat{u}(\beta(s)), \quad \lambda > 0.
\end{equation}

The principal symbol calculus encoded in \eqref{symbol equation of principal symbol of one fold -linearization} is indeed potential-free, since $V$ appears neither in the principal symbol nor the subprincipal symbol of $\Box + V$. To engage $V$, we must look into lower order symbols. 

\subsection{Lower order symbol calculus}\label{sec: lower order symbol calculus}

For a solution $u$ to \eqref{eqn : linear wave with potential}, we decompose it into \[u := u^{\fre} + u^{\rem},\] where $u^{\fre}$ solves the free wave equation
\begin{equation*}
	\Box u^{\fre} =f
\end{equation*} and $u^{\rem}$ is the remainder. Then it is clear from Proposition \ref{prop : linear parametrix} that 
\begin{equation*}
	\left\{\begin{aligned}
		u^{\fre} &\in I^{\mu}(\Lambda_0,\Lambda_1) \quad \mbox{with $\sigma\l[u^{\fre}\r]  = \sigma[u]$};\\  u^{\rem} &\in I^{\mu-1}(\Lambda_0,\Lambda_1).
	\end{aligned}  \right.
\end{equation*}

In virtue of the identities
$$
(\Box+V)u^{\rem} = (\Box+V)u-\Box u^{\fre}-Vu^{\fre}=-Vu^{\fre},
$$
the principal symbol of $u^{\rem}$ satisfies the following transport equation  
\begin{equation*}
	\begin{aligned}
		\begin{cases} 
			\mathscr{L}_{H_\Box}\sigma\l[u^{\rem}\r] = -\imath V\sigma[u]\quad &\text{on }\Lambda_1\b\partial\Lambda_1,\\
			\sigma\l[u^{\rem}\r]  = - V \sigma[\Box]^{-1}\sigma[u]\quad &\text{on } \Lambda_0\b\partial\Lambda_1,\\
			\sigma\l[u^{\rem}\r] = -V\mathscr{R}\l(\sigma[\Box]^{-1}\sigma[u] \r)\quad&\text{on }\partial\Lambda_1.
		\end{cases}
	\end{aligned}
\end{equation*}
Denoting $\sigma[u] := \hat{u}\omega$ and $\sigma\l[u^{\rem}\r] := \hat{u}^{\rem}\omega$, we convert this transport equation into
$$
\l(H_\Box \hat{u}^{\rem}\r) \omega  +\frac{1}{2}\l(\div H_\Box\r)\hat{u}^{\rem}\omega = -\imath V \hat{u}\omega.
$$
After restricting to the bicharacteristic curve $\beta(s) := (\gamma(s), \dot{\gamma}^\flat(s))$ and  removing the half-density $\omega$, we obtain 
$$
H_\Box \hat{u}^{\rem}(\beta(s)) + \frac{1}{2}\l(\hat{u}^{\rem}\div H_\Box \r)(\beta(s)) = -\imath V(\gamma(s)) \hat{u}(\beta(s)).
$$
It follows that
$$
\partial_{s}\l(e^{\rho(s)}\hat{u}^{\rem}(\beta(s))\r) = -\imath V(\gamma(s))e^{\rho(s)}\hat{u}(\beta(s)) ,
$$
which is solved by 
\begin{align} \label{principal symbol of the sub-leading term with term near source with out half-density}
	e^{\rho(s)}\hat{u}^{\rem}(\beta(s)) &= \hat{u}^{\rem}(\beta(0))-\imath \int_{0}^{s}V(\gamma(\tilde{s}))e^{\rho(\tilde{s})}\hat{u}(\beta(\tilde{s}))d\tilde{s}
	\\ 
	\notag	&=    \hat{u}^{\rem}(0,\c{\xi}) 
	-\imath \hat{u} (0,\c{\xi}) \int_{0}^{s}V(\gamma(\tilde{s}))    d\tilde{s},
\end{align}
with initial value given by \[ 
\sigma\l[u^{\rem}\r](0, \c{\xi}) =  -V(0)\mathscr{R}\l(\sigma[\Box]^{-1}\sigma[u]\r)(0,\c{\xi}).\]
Plugging \eqref{eqn: factor alpha as a zero-th order half-density} and \eqref{eqn: principal symbol along bicharacteristics} into \eqref{principal symbol of the sub-leading term with term near source with out half-density} shows
\begin{equation}\label{principal symbol of the sub-leading term with term near source}
\sigma[u^{\rem}](\beta(s)) = \alpha(\beta(s))\sigma[u^{\rem}](0,\c{\xi})-\imath\sigma[u](\beta(s))\int_{0}^{s}V(\gamma(\tilde{s}))d\tilde{s}.
\end{equation}

The first term on the RHS of \eqref{principal symbol of the sub-leading term with term near source} turns out to be of lower order.

\begin{proposition} \label{lemma: properties of principal symbol of sub-leading term of one-fold linearization}
	Away from $\partial\Lambda_1$,  we have  $$\sigma\l[u^{\rem}\r](\beta(s)) \in S^{\mu}(\Lambda_1\b\partial\Lambda_1),$$ and explicitly it reads
	\begin{align*}
		\sigma\l[u^{\rem}\r](\beta(s)) = -\imath\sigma[u](x,\xi)\int_{0}^{s}V(\gamma(\tilde{s}))d\tilde{s}   \mod S^{\mu-1}(\Lambda_1\b\partial\Lambda_1).
	\end{align*}

\end{proposition}

\begin{proof}
	Proposition \ref{prop : linear parametrix} guarantees  that $u\in I^{\mu}(\Lambda_0,\Lambda_1)$. By \cite[Proposition 4.1]{MU}, the principal symbol of $u$ satisfies
	\begin{align*}
		\sigma[u] \big|_{\Lambda_0\b\partial\Lambda_1}&\in S^{\mu+1/2}(\Lambda_0\b\partial\Lambda_1)\\ 
		\sigma[u] \big|_{\Lambda_1\b\partial\Lambda_0}&\in S^{\mu+1}(\Lambda_1\b\partial\Lambda_1).
	\end{align*}
	Since $\sigma[\Box]$ is of order $2$ on $\Lambda_0\b\partial\Lambda_1$,  \cite[Theorem 3]{CLOP} shows that  
	\begin{align*}
		(\sigma[\Box]^{-1}\sigma[u])\big|_{\Lambda_0\b\partial\Lambda_1} &\in S^{\mu-\frac{3}{2}}(\Lambda_0\b\partial\Lambda_1)
		\\ \mathscr{R}\l(\sigma[\Box]^{-1}\sigma[u]\r) &\in S^{\mu-1}(\partial\Lambda_1).
	\end{align*}
	This implies that $\alpha(\beta(s))\sigma\l[u^{\rem}\r](0, \c{\xi})$ is of order $\mu - 1$ due to the fact that $\alpha$  in \eqref{eqn: factor alpha as a zero-th order half-density} is of order $0$.

	From \eqref{homogeneity of the principal symbol}, we know that $\sigma[u]$ is of order $\mu+1$ on $\Lambda_1 \setminus \p \Lambda_1$. Also noting that $\int_{0}^{s}V(\gamma(\tilde{s}))\,d\tilde{s}$ is positively homogeneous of degree $-1$,
	we see that $\sigma\l[u^{\rem}\r](\beta(s))$ is of order $\mu$.\end{proof}

\section{Nonlinear conormal waves}

In this section, we shall elucidate the special solutions, which are of great use for inversion of potential for \eqref{eqn : semilinear model}.

In contrast with linear waves, nonlinear waves can propagate anomalously, i.e. the singularities of nonlinear waves propagate not only along the bicharacteristic curves. 
However, Bony \cite{Bony80, Bony82}, Rauch-Reed \cite{RR82}, and Melrose-Ritter \cite{MR-1985} proved for conormal waves that  the spreading of such singularities is limited to the production of conormal singularities. Such phenomena can be exploited to construct special solutions for inversion.

Conceptually, we would like to make special nonlinear waves such that 
\begin{itemize}
	\item the sources are small and made within $\mho$;
	\item the nonlinear waves are approximated by linear waves;
	\item the waves travel out of $\mho$ and return back;
	\item the nonlinear scattering in $\D \setminus \mho$ can be observed in $\mho$.
\end{itemize}
Then we can show that some scattering transforms encoded in such special solutions recover the desired coefficients in certain stable ways. This technique was pioneered by \cite{KLU}, where $4$ linear waves were utilized to approximate quadratic nonlinearity. The strategy in the present paper however invokes $3$ linear waves for  \eqref{eqn : semilinear model}.  This was formulated in \cite{CLOP} for cubic equations and applied to more general cases in \cite{FLO, FO}. 
In contrast with the principal symbol calculus and qualitative methods therein, we conduct symbol calculus, not only for principal terms but also for remainder terms, to extract the quantitative information of the potentials.

\subsection{Construction of conormal waves}
\label{sec : nonlinear waves}

To proceed, we first set up some notations 
 \begin{align*}
 	\mathbb{L} &:= \left\{  (x, y) \in \mathbb{D}^2 : \mbox{there is a future-pointing light-like line joining $x$ and $y$}  \right\};  \\ 
 	\mathbb{S}^+(\mho) &:= \left\{  (x, y, z) \in \mathbb{D}^3 : (x, y), (y, z) \in \mathbb{L},\, x, z\in \mho,\,\, y \notin \mho \right\}.
 \end{align*}
 For any $(\c{x}, \c{y}, \c{z}) \in \mathbb{S}^+(\mho)$,
 \begin{itemize}

 	\item we fix light-like lines $\gamma_{\mathrm{out}}$ and $\gamma_{(1)}$ connecting $\c{x}, \c{y}, \c{z}$ such that \begin{align*}\c{x} &= \gamma_{(1)}(0)  \\  \c{y} &= \gamma_{\mathrm{out}}(0) = \gamma_{(1)}(s_{\mathrm{in}})  \\ \c{z} &= \gamma_{\mathrm{out}}(s_{\mathrm{out}}) \end{align*} provided some $s_{\mathrm{in}}>0$ and $s_{\mathrm{out}} > 0$;
 	\item we require that the cotangent vector $(\dot{\gamma}_{(1)})^\flat$ is collinear with 
 	\[\c{\nu}_{(1)} := (-1, 1, 0, 0)\] and the cotangent vector $(\dot{\gamma}_{\mathrm{out}})^\flat$ is collinear with 
 	\[\c{\eta} := (-1,  -a(r_0), r_0, 0), \quad \mbox{for some $r_0 \in [-1, 1]$}\] given  the function $a(r) := \sqrt{1 - r^2}$ for $0 \leq r < 1$. \end{itemize}
We write $\c{x}_{(1)} : = \c{x}$ and find $\c{x}_{(k)} \in \mho, k = 2, 3$ so that \begin{itemize} 
  \item $\c{x}_{(1)},\c{x}_{(2)}$ and $\c{x}_{(3)}$ are causally independent, that is, $\c{x}_{(j)}$ does not lie on the causal future of $\c{x}_{(k)}$, for $1\le j\neq k\le 3$;

\item there is a light-like line $\gamma_{(k)}$ such that $$\gamma_{(k)}(0) = \c{x}_{(k)}\quad \mbox{and} \quad \gamma_{(k)}(s_{\mathrm{in}}) = \c{y};$$

 	\item  $(\dot{\gamma}_{(k)}(s_{\mathrm{in}}))^\flat$ is collinear with  
 	\[\c{\nu}_{(k)} := (-1, a(r), (-1)^{k}r, 0), \quad k = 2, 3.\]

   \item For convenience, we denote 
   $$
    \c{\xi}_{(1)} = -\c{\nu}_{(1)},\quad\c{\xi}_{(2)} = \c{\nu}_{(2)},\quad\c{\xi}_{(3)} = \c{\nu}_{(3)}.
   $$
    
    \end{itemize}  
\cite[Lemma 1]{CLOP} shows that there exist some $1/2 \leq \kappa_{(j)} \leq 9/2$ such that
 \begin{equation}\label{eqn: linear relation of three light-like covectors}
 	\c{\eta} = r^{-2} \kappa_{(1)} \c{\xi}_{(1)}   +  r^{-2} \kappa_{(2)} \c{\xi}_{(2)}   +  r^{-2} \kappa_{(3)} \c{\xi}_{(3)}.
 \end{equation}


We make sources \begin{align}\label{eqn : source terms}f_{(k)} := \boldsymbol{\chi}_{(k)} \delta_{\c{x}_{(k)}}, \quad k = 1, 2, 3\end{align} where $\boldsymbol{\chi}_{(k)} $ and $\delta_{\c{x}_{(k)}}$ satisfy the following conditions: 
 \begin{itemize}\item $\delta_{(k)}$ for $k= 1, 2, 3$ is the Dirac delta distribution at $\c{x}_{(k)}$; \item $\boldsymbol{\chi}_{(1)}$ is a microlocal cut-off near $(\c{x}_{(1)}, -(\dot{\gamma}_{(1)})^\flat)$; \item $\boldsymbol{\chi}_{(j)}$ for $j = 2 ,3$ is a microlocal cut-off near $(\c{x}_{(j)}, (\dot{\gamma}_{(j)})^\flat)$; \item the principal symbol $\sigma[\boldsymbol{\chi}_{(k)}]$ is positively homogeneous of degree $\mu+5/2$ with $\mu=-s-3/2$;
 	\item the support of $f_{(k)}$ is contained in a sufficiently small neighbourhood $\mho_{(k)}$ of $\c{x}_{(k)}$; 
 	\item $\|f_{(k)}\|_{\mathscr{H}^s(\mho)} = 1$ for $k = 1, 2, 3.$
 \end{itemize}
 
 For sufficiently small $\e_{(j)} > 0$, we denote $\e := (\e_{(1)}, \e_{(2)}, \e_{(3)})$ and consider the source
 \begin{align}\label{eqn : fcub}
\e_{(1)}f_{(1)}+\e_{(2)}f_{(2)}+\e_{(3)}f_{(3)}.
 \end{align}
 By \cite[Theorem 6.5.2]{Rauch-book}, the local solution $u(\e)$ to \eqref{eqn : semilinear model} with source \eqref{eqn : fcub}  depends smoothly on $\e$.
 
We may view $u(\e)$ as a smooth function in $\e$ and make Taylor's expansion within $\mho$ near $\e = 0$ up to $O(|\e|^4)$. Denote the partial  derivatives of $u(\e)$  at  $\e = 0$ by 
 \begin{align*}
 	u_{(j)} &:= \partial_{\e_{(j)}}u|_{\e=0} \\ 
 	u_{(jk)} &:= -\frac{1}{2!}\partial^2_{\e_{(j)}\e_{(k)}}u|_{\e=0} \\ 
 	u_{(jkl)} &:= -\frac{1}{3!}\partial^3_{\e_{(j)}\e_{(k)}\e_{(l)}}u|_{\e=0}\end{align*} with $j, k, l \in \{1, 2 ,3\}$. Since $u|_{\e = 0} = 0$, these derivatives  solve the following system of wave equations
 \begin{align}
 	\label{one-fold linearization}	\l(\Box +V\r)u_{(j)} &= f_{(j)}; \\ 
 	\label{two-fold linearization}		\l(\Box +V\r)u_{(jk)} &= 0; \\ 
 	\label{three-fold linearization}		\l(\Box +V\r) u_{(jkl)} &= h u_{(j)}u_{(k)}u_{(l)}.
 \end{align} In particular,   \eqref{two-fold linearization}  implies $u_{(jk)} \equiv 0$.


Moreover, the input \eqref{eqn : fcub} guarantees that the $u_{(jkl)}$-terms with distinct $j, k, l$ are dominated quantitatively by the $u_{(j)}$-terms and $u_{(123)}$-term in the Taylor's expansion.  To see this, it requires  Greenleaf--Uhlmann's calculus of conormal distributions \cite{GU}. \begin{proposition}\label{prop : calculus of conormal}
 	Let $Y_j \subset X$ for $j = 1, 2$ be transversal $p$-submanifolds and $$u_j = \hat{u}_j \omega\in I^{\mu_j}(N^\ast Y_j)$$  conormal distributions with a nowhere vanishing smooth half-density  $\omega$. \begin{itemize} \item Then the product distribution $u_1 u_2$ is an element of 
 	\[I^{\mu_1}(N^\ast (Y_1 \cap Y_2), N^\ast Y_1) + I^{\mu_2}(N^\ast (Y_1 \cap Y_2), N^\ast Y_2),\] 
 	where the product of conormal half-densities is defined as 
 	\[u_1 u_2 :=    \hat{u}_1   \hat{u}_2  \omega.\]   
 \item	Microlocally away from $N^\ast Y_1 \cup N^\ast Y_2$,  the product distribution $u_1 u_2$ is a conormal distribution lying in $$I^{\mu_1 + \mu_2 + n/4}(N^\ast (Y_1 \cap Y_2))$$ and having a principal symbol 
 	\[\sigma[u_1 u_2]\big|_{N^\ast (Y_1 \cap Y_2)} =   \hat{u}_1\big|_{N^\ast Y_1}  \hat{u}_2\big|_{N^\ast Y_2} \omega .\] \end{itemize} 
 \end{proposition}

By Proposition \ref{prop : linear parametrix}, each   $u_{(j)}, j = 1, 2, 3$ is a conormal distribution such that 
 \[      \mathrm{WF}(u_{(j)})   \subset N^\ast K_{(j)},\] 
 where $N^\ast K_{(j)}$ is the future flow-out of $\WF(f_{(j)})$.
 
 Since $\{K_{(j)} : j = 1, 2, 3\}$ are pairwisely transverse, 
 Proposition \ref{prop : calculus of conormal} further tells us that the product distribution $u_{(i)} u_{(j)} u_{(k)}$ has a wavefront set contained in
 \[N^\ast \l(K_{(i)} \cap K_{(j)} \cap K_{(k)} \r) \cup \l(\bigcup_{\{j , k\} \subset \{1, 2, 3\}} N^\ast \l(K_{(j)} \cap K_{(k)}\r)\r) \cup \l(\bigcup_{l \in \{1, 2, 3\}} N^\ast K_{(l)}  \r).\] 
 
 Since the sum of two light-like covectors cannot be light-like and $$N^\ast_x (K_{(j)} \cap K_{(k)}) = N^\ast_x  K_{(j)}  \oplus N^\ast_x  K_{(k)},$$ we see that 
 $$
 \l(N^\ast (K_{(j)} \cap K_{(k)})\b \l(N^{\ast}K_{(j)}\cup N^{\ast}K_{(k)} \r)\r) \cap \{\sigma[\Box] = 0\} = \emptyset. 
 $$
 Propagation of singularities guarantees that  the part of the product associated with $N^\ast \l(K_{(j)} \cap K_{(k)}\r) \b \l(N^{\ast}K_{(j)}\cup N^{\ast}K_{(k)} \r) $ exhibits no singularities in $\mho$. 

 Aside from the linear waves, one can only observe in $\mho$ the singularity of $h u_{(1)} u_{(2)} u_{(3)}$ associated with
 \begin{equation*}\label{eqn: conormal bundle of the interaction wave}
 \Lambda_0 := N^\ast \left( K_{(1)}   \cap K_{(2)}  \cap K_{(3)} \right) \setminus 0.
 \end{equation*}
 The third order term $u_{(123)}|_{\mho}$ is the wave propagating from $h u_{(1)} u_{(2)} u_{(3)}$.
 Consequently, the cubic waves, restricted within $\mho$, admit the Taylor's expansion 
 \begin{align}\label{eqn : taylor of u}
 	u|_{\mho} = \sum_{j=1}^3 u_{(j)} \e_{(j)} - u_{(123)} \e_{(1)} \e_{(2)} \e_{(3)} + O(|\e|^4) \mod C^\infty(\rotom).\end{align}

 To understand $u|_{\mho}$ from the perspective of conormal distributions, we first employ Proposition \ref{prop : linear parametrix} to construct the first-order terms $u_{(j)}$ in \eqref{eqn : taylor of u}. Since $u_{(j)}$ obeys \eqref{one-fold linearization} with  $u_{(j)} \in I(N^\ast K_{(j)})$, the product $u_{(1)} u_{(2)} u_{(3)}$ thus lies in the space 
 \[I\left(\Lambda_0\right)  + \sum_{\{j, k\} \subset \{1, 2, 3\}}  I\l(N^\ast \l(K_{(j)} \cap K_{(k)}\r), N^\ast K_{(j)}\r)+ \sum_{j=1}^3 \mathcal{D}'\left( N^\ast K_{(j)}, \e \right)\] where $\mathcal{D}'\left( N^\ast K_{(j)}, \e \right)$ is the space of distributions with wavefront set contained in a small conic neighbourhood of $N^\ast K_{(j)}$. See \cite[Section 3.3]{LUW} for more details.

 
 \subsection{Principal symbols}\label{sec : principal singularities}
 
 To separate the singular parts from smooth parts, we may choose a pseudodifferential operator $\boldsymbol{\chi}$ such that 
 \begin{equation} \label{eqn : microlocal cut-off at y}
 \left\{	\begin{aligned}&
 \mbox{ $\WF(\boldsymbol{\chi})$ lies in a sufficiently small conic neighbourhood of $(\c{y},\c{\eta}) \in \Lambda_0$;} \\&\mbox{ $\WF(\boldsymbol{\chi})$ is away from all $N^\ast K_{(j)}$ and $N^\ast K_{(kl)}$.}\end{aligned}\right. \end{equation}
With the shorthand notation
\begin{equation}\label{source of the interaction wave of three product waves}
 	f_{(123)} := h u_{(1)} u_{(2)}u_{(3)},
\end{equation} 
 we invoke the symbol calculus in Proposition \ref{prop : calculus of conormal} to compute
 	\begin{align}\notag 
 		\sigma\l[f_{(123)}\r](\c{y},\c{\eta})
  &=   h(\c{y})\prod_{j=1}^{3} \sigma\l[u_{(j)}\r]\l(\c{y},r^{-2}\kappa_{(j)}\c{\xi}_{(j)}\r)\\
 	\notag	&= r^{-6(\mu+1)}  h(\c{y})\prod_{j=1}^{3}\kappa_{(j)}^{\mu+1} \sigma\l[u_{(j)}\r]\l(\c{y},\c{\xi}_{(j)}\r)\\
 	\notag	&= r^{-6(\mu+1)}h(\c{y})\l(\prod_{j=1}^{3}\kappa_{(j)}^{\mu+1}\hat{u}_{(j)}\l(\c{y},\c{\xi}_{(j)}\r)\r) \omega.
 	\end{align} 
 For brevity, we will denote the product in the parentheses by
 \[\alpha_{\into}\l(\c{y},\c{\eta};\c{X},\c{\Xi}\r) := \prod_{j=1}^{3}\kappa_{(j)}^{\mu+1}\hat{u}_{(j)}\l(\c{y},\c{\xi}_{(j)}\r) \] 
 with $\c{X} := \l(\c{x}_{(1)},\c{x}_{(2)},\c{x}_{(3)}\r)$ and $\c{\Xi} := \l(\c{\xi}_{(1)},\c{\xi}_{(2)},\c{\xi}_{(3)}\r)$.

  We then split $u_{(123)}$ into $$u_{(123)} = v_{(123)}+w_{(123)},$$  where the two components solve the following two wave equations respectively
 \begin{align}\label{decomposition of three-fold linearization}
 	(\Box+V)   v_{(123)}&=\boldsymbol{\chi} f_{(123)} 
 	\\
 	\notag	(\Box+V)   w_{(123)}&=(\id-\boldsymbol{\chi}) f_{(123)}.
 \end{align}
 As designed, the characteristic variety of $\Box$  does not  intersect with $\WF(\id - \boldsymbol{\chi})$ near $\gamma_{\out}$. Propagation of singularities yields that $w_{(123)}$ must be smooth near $\gamma_{\text{out}}$.

We next analyse the singularities of $v_{(123)}$ in \eqref{decomposition of three-fold linearization}, and compute the principal symbol of $v_{(123)}$ in $\rotom$.

Denote by $\Lambda_1$ the future flow-out of  $\Lambda_0. $  Proposition \ref{prop : linear parametrix} yields  the solution $$v_{(123)}\in I^{3\mu+1/2}(\Lambda_0,\Lambda_1)$$ and the transport equation for $\sigma[v_{(123)}]$ along $\Lambda_1$
 \begin{equation} \label{transport equation of principal term of three interaction waves}
 	\begin{aligned}
 		\begin{cases}
 			\mathscr{L}_{H_\Box}\l(\sigma\l[v_{(123)}\r]\r) = 0 \quad &\text{ on } \Lambda_1\backslash \partial\Lambda_1 ,\\
 			\sigma\l[v_{(123)}\r]  = \mathscr{R}\l(\sigma\l[\Box\r]^{-1} \sigma\l[\boldsymbol{\chi} f_{(123)}\r]\r) \quad &\text{ on } \partial\Lambda_1.
 		\end{cases}
 	\end{aligned}
 \end{equation}
Similar to \eqref{integral factor of the transport equation and half-density} and   \eqref{eqn: factor alpha as a zero-th order half-density}, we denote a strictly positive factor as in
 $$
 \alpha_{\out}(\beta_{\out}(s)):=e^{-\rho_{\out}(s)}\omega(\beta_{\out}(s))\omega^{-1}(\beta_{\out}(0)),
 $$
 with bicharacteristic curve $\beta_{\out}\subseteq \Lambda_1$ and integral factor 
 $$
 \rho_{\out}(s):=\frac{1}{2}\int_{0}^{s}\div H_\Box(\beta_{\out}(t))dt. 
 $$ 
 The solution to \eqref{transport equation of principal term of three interaction waves} at $(\c{z},\c{\zeta})$ is given via \eqref{eqn: principal symbol along bicharacteristics} by
   \begin{align}
 \label{parallel transport along y,eta to z,zeta}		\sigma\l[v_{(123)}\r](\c{z},\c{\zeta})  
 		 &= \alpha_{\out}(\c{z},\c{\zeta})\sigma\l[v_{(123)}\r](\c{y},\c{\eta})  \\   \notag &=r^{-6(\mu+1)}h(\c{y})\l(\alpha_{\out}\alpha_{\into}\mathscr{R}\l(\sigma[\Box]^{-1}\omega\r)\r)\l(\c{z},\c{\zeta};\c{y},\c{\eta};\c{X},\c{\Xi}\r).
    \end{align} 
   Denoting the scalar factor in \eqref{parallel transport along y,eta to z,zeta} by
    \begin{equation}\label{principal symbol of the interaction wave of three waves}  
    \mathcal{A}_{(123)} := \l(\alpha_{\out}\alpha_{\into}\mathscr{R}\l(\sigma[\Box]^{-1}\omega\r)\r)\l(\c{z},\c{\zeta};\c{y},\c{\eta};\c{X},\c{\Xi}\r),
  \end{equation}
 we have that $\mathcal{A}_{(123)}$ is a non-vanishing factor independent of $h$ and $v$, and bounded from below and above
 \begin{equation}\label{uniform bound of A}
  \C_{(123)}^{-1} \le |\A_{(123)}| \le \C_{(123)},
 \end{equation}
for some constant $\C_{(123)} > 0$. Moreover,  \begin{equation}\label{homogeneity of interaction wave of three waves}
 	\sigma\l[v_{(123)}\r](\c{z},t\c{\zeta}) = t^{3\mu+\frac{3}{2}}\sigma\l[v_{(123)}\r](\c{z},\c{\zeta}),\quad \text{for all } t\ge 1.
 \end{equation}
  
Therefore, we have proved
 \begin{lemma}
The principal symbol of $v_{(123)}$ at $(\c{z},\c{\zeta})$ takes the form
 \begin{equation}\label{principal symbol of v_123}
\sigma\l[v_{(123)}\r](\c{z},\c{\zeta})  =   	r^{-6(\mu+1)}h(\c{y})\mathcal{A}_{(123)} \l(\c{z},\c{\zeta};\c{y},\c{\eta};\c{X},\c{\Xi}\r)
 \end{equation}
 satisfying \eqref{parallel transport along y,eta to z,zeta}, \eqref{principal symbol of the interaction wave of three waves}, \eqref{uniform bound of A} and \eqref{homogeneity of interaction wave of three waves}.
\end{lemma}

\subsection{Lower order symbols} \label{sec : lower order symbol}

In order to recover $V$, it is required to analyse the lower-order  terms of $u_{(123)}$ separately, as it proceeds for linear waves. We write $$f_{(123)}:= f_{(123)}^{\fre} + f_{(123)}^{\rem} $$  
 where 
 \begin{align}
 \notag	f_{(123)}^{\fre} &:= h u^{\fre}_{(1)}u^{\fre}_{(2)}u^{\fre}_{(3)}\\
 \label{eqn: free source and remainder source}	f_{(123)}^{\rem} &:= \sum_{\{i,j,k\} = \{1,2,3\}} \l(h u_{(i)}^{\rem} u_{(j)}^{\fre} u_{(k)}^{\fre} + h u_{(i)}^{\rem} u_{(j)}^{\rem} u_{(k)}^{\fre}\r)+hu_{(1)}^{\rem}u_{(2)}^{\rem}u_{(3)}^{\rem}.
 \end{align} 
Proposition \ref{lemma: properties of principal symbol of sub-leading term of one-fold linearization} shows that
 $ u_{(j)}^{\fre} \in I^{\mu}(N^{\ast}K_{(j)})$ and $u_{(j)}^{\rem} \in I^{\mu-1}(N^{\ast}K_{(j)})$.
 The application of   $\boldsymbol{\chi}$ in \eqref{eqn : microlocal cut-off at y}  to them respectively leads to
 \begin{align}
 	\notag \boldsymbol{\chi} f_{(123)}^{\fre} &\in I^{3\mu+2}\l(\Lambda_0\r), && \sigma\l[\boldsymbol{\chi} f_{(123)}^{\fre}\r]  = \sigma\l[\boldsymbol{\chi} f_{(123)}\r];
 	\\
 	\label{sub-leading term of three interaction waves modulo lower order terms}
 	\boldsymbol{\chi} f_{(123)}^{\rem}   &\in I^{3\mu+1}(\Lambda_0),
 	&&\sigma\l[\boldsymbol{\chi} f_{(123)}^{\rem}\r] \in S^{3\mu+2}(\Lambda_0)/S^{3\mu+1}(\Lambda_0).
 \end{align}

 Let $u_{(123)}^{\fre}$ be the free wave such that
 \begin{equation}\label{eqn: free wave for three-fold linearization}
 	\Box u_{(123)}^{\fre} = h u_{(1)}^{\fre}u_{(2)}^{\fre}u_{(3)}^{\fre} = f_{(123)}^{\fre} 
 \end{equation}
and write $$u_{(123)} := u_{(123)}^{\fre}+u_{(123)}^{\rem}.$$ It then follows that
 \begin{align*}
 f_{(123)} =	(\Box+V) \l( u_{(123)}^{\fre} + u_{(123)}^{\rem}\r) 
 	&= f_{(123)}^{\fre} + Vu_{(123)}^{\fre} + (\Box +V) u_{(123)}^{\rem},
 \end{align*}
which amounts to that
 $$
 (\Box+V)u_{(123)}^{\rem} = f_{(123)}^{\rem} -Vu_{(123)}^{\fre}.
 $$ 
Consider coupled linear wave equations with source  $\boldsymbol{\chi} f_{(123)}^{\fre}$ and $\boldsymbol{\chi}f_{(123)}^{\rem}$
 \begin{align}
 \label{the equation of the principal term}	\Box v_{(123)}^{\fre} &= \boldsymbol{\chi} f_{(123)}^{\fre} ,\\  
\label{eqn : v123rem} 	(\Box+V)v_{(123)}^{\rem} &= \boldsymbol{\chi} f_{(123)}^{\rem}  -Vv_{(123)}^{\fre}.
 \end{align}
Again due to propagation of singularities, both $w_{(123)}^{\fre} := u_{(123)}^{\fre} - v_{(123)}^{\fre}$ and $w_{(123)}^{\rem} := u_{(123)}^{\rem} - v_{(123)}^{\rem}$ exhibit no singularity around $\gamma_{\mathrm{out}}$.

Noting $\sigma\l[v_{(123)}^{\fre}\r] = \sigma\l[v_{(123)}\r]$, we have 
 \begin{equation}\label{the principal symbol of v_123^pri}
 	v_{(123)}^{\fre} \in I^{3\mu+\frac{1}{2}}\l(\Lambda_0,\Lambda_1\r).
 \end{equation}

 On the other hand, we can compute the principal symbol of $v_{(123)}^{\rem}$ in $\rotom$.
\begin{lemma}
 The principal symbol of $v_{(123)}^{\rem}$ at $(\c{z},\c{\zeta})$ is
  \begin{multline}\label{principal symbol of the sub-leading term of three interaction waves}
 	\sigma\l[v_{(123)}^{\rem}\r](\c{z},\c{\zeta})   = -\imath \sigma\l[v_{(123)}\r](\c{z},\c{\zeta})\l( \int_{\gamma_{\out}} V  + \sum\limits_{i=1}^{3}  r^2|\kappa_{(i)}|^{-1}\int_{\gamma_{(i)}} V \r)\\
   = -\imath r^{-6(\mu+1)} \mathcal{A}_{(123)}\l(\c{z},\c{\zeta};\c{y},\c{\eta},\c{X},\c{\Xi}\r)h(\c{y}) \l( \int_{\gamma_{\out}}V + O(r^{2})\r), \quad\text{as }r\to 0,
 \end{multline} 
 where the remainder $O(r^{2})$ has a uniform upper bound $4 T \| V \|_{L^{\infty}(\overline{\D})} r^2$. 
\end{lemma}

 \begin{proof}
 By comparing the orders in \eqref{sub-leading term of three interaction waves modulo lower order terms} and \eqref{the principal symbol of v_123^pri}, we see that  the principal symbol of the RHS of \eqref{eqn : v123rem}  is equal to $\sigma\l[\boldsymbol{\chi} f_{(123)}^{\rem}\r]$ on $\Lambda_0\b\partial\Lambda_1$ and $V\sigma\l[v_{(123)}\r]$ on $\Lambda_1\b\partial\Lambda_1$. In other words, $\sigma\l[v_{(123)}^{\rem}\r]$ solves the following transport equation
 \begin{equation}
 	\begin{aligned}
 		\begin{cases}
 			\mathscr{L}_{H_\Box}\sigma\l[v_{(123)}^{\rem}\r] = -\imath V \sigma\l[v_{(123)}\r],\quad &\text{on }\Lambda_1\b\partial\Lambda_1,\\
 			\sigma\l[v_{(123)}^{\rem}\r] = \sigma[\Box]^{-1}\sigma\l[\boldsymbol{\chi} f_{(123)}^{\rem}\r],\quad &\text {on }\Lambda_0\b\partial\Lambda_1,\\
 			\sigma\l[v_{(123)}^{\rem}\r] = \mathscr{R}\l(\sigma[\Box]^{-1}\sigma\l[\boldsymbol{\chi} f_{(123)}^{\rem}\r]\r),\quad &\text {on }\partial\Lambda_1.
 			\nonumber
 		\end{cases}
 	\end{aligned}
 \end{equation}

Proceeding verbatim as for \eqref{principal symbol of the sub-leading term with term near source}, we derive the solution 
 \begin{equation} \label{solution of the transport equation of v_123^sub}
 		\sigma\l[v_{(123)}^{\rem}\r](\c{z},\c{\zeta})  =  -\imath\sigma\l[v_{(123)}\r](\c{z},\c{\zeta})\int_{\gamma_{\out}} V   
 	  + \alpha_{\out}(\c{z},\c{\zeta})\sigma\l[v_{(123)}^{\rem}\r](\c{y},\c{\eta}) ,   
 \end{equation}
with initial value 
 \[	\sigma\l[v_{(123)}^{\rem}\r](\c{y},\c{\eta}) =  
 \mathscr{R}\l(\sigma[\Box]^{-1}\sigma\l[\boldsymbol{\chi} f_{(123)}^{\rem}\r]\r)(\c{y},\c{\eta}),\]where we use the shorthand $$\int_{\gamma_{\out}} V  := \int_{0}^{s_{\out}}V(\gamma_{\out}(s))ds.$$
 
With regard to $\sigma\l[\boldsymbol{\chi} f_{(123)}^{\rem}\r](\c{y},\c{\eta})$, observe that the leading terms of $\boldsymbol{\chi} f_{(123)}^{\rem}$ consist only of type $h u_{(i)}^{\rem}u_{(j)}^{\fre} u_{(k)}^{\fre}$. We calculate the principal symbols for this type and assert 
\begin{equation}\label{principal symbol of the subleading term of the interaction sources}
	\sigma\l[\boldsymbol{\chi} f_{(123)}^{\rem}\r](\c{y},\c{\eta}) = -\imath\sigma\l[ \boldsymbol{\chi} f_{(123)}\r](\c{y},\c{\eta})  \sum\limits_{i=1}^{3}  r^2 \kappa_{(i)}^{-1}\int_{\gamma_{(i)}}V,
\end{equation} 
where we  write 
\begin{align*}
\int_{\gamma_{(i)}}V &:= \int_{0}^{s_{\into}}V(\gamma_{(i)}(s))ds, \quad \mbox{for $i=2,3,$}\\ \int_{\gamma_{(1)}}V &:= \int_{0}^{-s_{\into}}V(\gamma_{(1)}(-s))ds.
\end{align*}
We remark that the sign change in $\int_{\gamma_{(1)}} V$ is due to that the bicharacteristic flow $\beta_{(1)}$ from $(\c{x}_{(1)}, -(\dot{\gamma}_{(1)})^\flat)$ to $(\c{y},\c{\xi}_{(1)})$ is 
$$
\beta_{(1)}(s) = (\gamma_{(1)}(-s),-(\dot{\gamma}_{(1)})^{\flat}(-s)),\quad s\in [-s_{\into},0].
$$ 
In fact, Proposition \ref{lemma: properties of principal symbol of sub-leading term of one-fold linearization} enables the following symbol calculation to prove \eqref{principal symbol of the subleading term of the interaction sources}.  Denoting $\lambda_{(i)} := r^{-2}\kappa_{(i)}$, we get
 	\begin{align*}
 		\lefteqn{\sigma\l[\boldsymbol{\chi}\l(h u_{(i)}^{\rem}u_{(j)}^{\fre} u_{(k)}^{\fre}\r)\r](\c{y},\c{\eta})} \\
 		&= h(\c{y})\sigma \l[u_{(i)}^{\rem}\r]\l(\c{y},\lambda_{(i)}\c{\xi}_{(i)}\r) \sigma\l[u_{(j)}^{\fre}\r]\l(\c{y},\lambda_{(j)}\c{\xi}_{(j)}\r)
 		\sigma\l[u_{(k)}^{\fre}\r]\l(\c{y},\lambda_{(k)}\c{\xi}_{(k)}\r)\\
 		&= \lambda_{(i)}^{\mu}h(\c{y})\sigma\l[u_{(i)}^{\rem}\r]\l(\c{y},\c{\xi}_{(i)}\r) \sigma\l[u_{(j)}\r]\l(\c{y},\lambda_{(j)}\c{\xi}_{(j)}\r)
 		\sigma\l[u_{(k)}\r]\l(\c{y},\lambda_{(k)}\c{\xi}_{(k)}\r)\\
 		&=  \lambda_{(i)}^{\mu}h(\c{y})\sigma\l[u_{(i)}\r]\l(\c{y},\c{\xi}_{(i)}\r) \sigma\l[u_{(j)}\r]\l(\c{y},\lambda_{(j)}\c{\xi}_{(j)}\r)
 		\sigma\l[u_{(k)}\r]\l(\c{y},\lambda_{(k)}\c{\xi}_{(k)}\r)\int_{\gamma_{(i)}}V\\
 		&= \lambda_{(i)}^{-1}h(\c{y})\sigma\l[u_{(i)}\r]\l(\c{y},\lambda_{(i)}\c{\xi}_{(i)}\r) \sigma\l[u_{(j)}\r]\l(\c{y},\lambda_{(j)}\c{\xi}_{(j)}\r)
 		\sigma\l[u_{(k)}\r]\l(\c{y},\lambda_{(k)}\c{\xi}_{(k)}\r)\int_{\gamma_{(i)}}V\\
 		&=-\imath r^{2}\kappa_{(i)}^{-1}\sigma\l[\boldsymbol{\chi} f_{(123)}\r](\c{y},\c{\eta})\int_{\gamma_{(i)}}V.
 		\nonumber
 	\end{align*}

Using \eqref{principal symbol of the subleading term of the interaction sources}, we calculate the second term of the RHS of \eqref{solution of the transport equation of v_123^sub} 
 \begin{align}\label{principal symbol of v_123^sub at y,eta}
 		\lefteqn{\alpha_{\out}(\c{z},\c{\zeta})\mathscr{R}\l(\sigma[\Box]^{-1}\sigma\l[\boldsymbol{\chi} f_{(123)}^{\rem}\r]\r)(\c{y},\c{\eta})} \\
 \notag		=&-\imath \alpha_{\out}(\c{z},\c{\zeta})\mathscr{R}\l(\sigma[\Box]^{-1}\sigma\l[ \boldsymbol{\chi} f_{(123)}\r]\r)(\c{y},\c{\eta})\sum\limits_{i=1}^{3}  r^2\kappa_{(i)}^{-1}\int_{\gamma_{(i)}}V\\
 \notag =&-\imath \alpha_{\out}(\c{z},\c{\zeta})\sigma[v_{(123)}](\c{y},\c{\eta})\sum\limits_{i=1}^{3}  r^2\kappa_{(i)}^{-1}\int_{\gamma_{(i)}}V\\
 \notag		=& -\imath\sigma\l[v_{(123)}\r](\c{z},\c{\zeta})\sum\limits_{i=1}^{3}  r^2\kappa_{(i)}^{-1}\int_{\gamma_{(i)}}V.
 \end{align}

 Combining \eqref{principal symbol of v_123}, \eqref{solution of the transport equation of v_123^sub} and \eqref{principal symbol of v_123^sub at y,eta} together proves \eqref{principal symbol of the sub-leading term of three interaction waves}.\end{proof}

\subsection{Trilinear response operators}
\label{sec : trilinear operators}

For any $f \in \mathscr{B}^s_\varrho(\mho),$ the solution to \eqref{eqn : semilinear model} admits  the Taylor's expansion \eqref{eqn : taylor of u}.  Hence, the source-to-solution map $\P_{V, h}$, restricted to $\mathscr{B}^s_\varrho(\mho)$, takes the explicit form 
\begin{align*}\P_{V, h} : \mathscr{B}^s_\varrho(\mho) &\longrightarrow \mathscr{H}^s(\mho) \\ 
	f &\longmapsto \sum_{j=1}^3 u_{(j)} \e_{(j)} + \frac{1}{6} u_{(123)} \e_{(1)} \e_{(2)} \e_{(3)} + O(|\e|^4) + C^\infty(\rotom).  \end{align*} 

We want to measure the discrepancy of two source-to-solution maps $\P_{V, h}$ and $\P_{\tilde{V}, \tilde{h}}$ provided that $(V, h)$ and $(\tilde{V}, \tilde{h})$ agree within $\mho$. We denote by $\tilde{u}$ the solution with respect to $(\tilde{V}, \tilde{h})$.

Since $\mho$ is geodesically convex, the first order terms $u_{(j)}$ are linear conormal waves propagating within $\mho$ and thus carry no information from $\D \setminus \mho$. Since the underlying coefficients are known in $\mho$,  the first order terms with respect to $(V, h)$ and $(\tilde{V}, \tilde{h})$ must agree therein, modulo a smooth function, regardless of inputs. Namely, \[u_{(j)} |_{\mho} = \tilde{u}_{(j)} |_{\mho} \mod C^{\infty}(\rotom), \quad \forall f \in \mathscr{H}^s(\mho).\]

Therefore, it is  the third order term $u_{(123)}$ that serves as the leading term with information from $\D \setminus \mho$. It can be viewed the restriction to $(\mathscr{H}^{s}(\rotom))^3$ of a trilinear operator, which is defined by
\begin{align*}  
	T_{V, h} \l(f_{(1)},f_{(2)},f_{(3)}\r)  := u_{(123)}|_{\rotom}, \quad \mbox{for $f_{(j)} \in \mathscr{H}^s(\mho)$.}
\end{align*}

\begin{lemma}\label{lemma: tri-linear operator is bounded}
	The map  $T_{V, h}$ is a bounded trilinear operator, \begin{align*} 
		\mathscr{H}^s(\mho) \times \mathscr{H}^s(\mho) \times \mathscr{H}^s(\mho)  &\longrightarrow \mathscr{H}^s(\mho) \\
		\l(f_{(1)},f_{(2)},f_{(3)}\r)  &\longmapsto u_{(123)}|_{\rotom},
	\end{align*} with operator norm $$
	\l\| T_{V, h} \r\|    := \sup  \l\{ \| u_{(123)}\|_{\mathscr{H}^s(\mho)} : \left\|f_{(j)}\right\|_{\mathscr{H}^s(\mho)} = 1   \r\}.
	$$
\end{lemma}

\begin{proof} The trilinearity is immediate from \eqref{one-fold linearization} and \eqref{three-fold linearization}.
	Since $u_{(j)}$ solves \eqref{one-fold linearization}, the existence theorem for linear wave equations and finite speed propagation yields 
	\[  \|u_{(j)}\|_{\mathscr{H}^s(\D)}  \lesssim \|  f_{(j)} \|_{\mathscr{H}^s(\mho)}.  \]
	Moreover, applying the existence theorem to \eqref{three-fold linearization} and Schauder's lemma \cite[Theorem 6.2.2]{Rauch-book} to product $h u_{(1)} u_{(2)} u_{(3)}$ shows that
	\[  \|u_{(123)}\|_{\mathscr{H}^s(\mho)}  \lesssim \|  h u_{(1)} u_{(2)} u_{(3)}\|_{\mathscr{H}^s(\D)}  \lesssim \prod_{j = 1}^3 \|  f_{(j)} \|_{\mathscr{H}^s(\mho)} .\]
\end{proof}

\begin{lemma}\label{lemma: stability estimate of the tri-linear operator}
	Let the observation error $\delta > 0$ be small. If $\P_{V,h} - \P_{\tilde{V},\tilde{h}}$ satisfies  \eqref{error of the source-to-solution map V h} in the cubic semilinear model, 
	$$
	\l\| T_{V, h}-T_{\tilde{V}, \tilde{h}}\r\| \lesssim \delta^{\frac{2}{5}}.
	$$
\end{lemma}
\begin{proof}
Let $u$ be the solution to \eqref{eqn : semilinear model} with  source $f$ of the form \eqref{eqn : fcub},  and $u_{(j)}$  the solution of \eqref{one-fold linearization} with source $f_{(j)}$. It follows that
$$
(\Box+V)\l(u-\sum_{j=1}^{3}\e_{(j)}u_{(j)}\r) + hu^{3} = 0.
$$
Applying $Q := \l(\Box+V \r)^{-1}$  to this equation gives
$$
u = \sum_{j=1}^{3}\e_{(j)}u_{(j)}-Q(hu^{3}),
$$
which implies that
\begin{align*}
	u &= \sum_{j=1}^{3}\e_{(j)}u_{(j)}- Q \left(h\left(\sum_{j=1}^{3}\e_{(j)}u_{(j)}-Q(hu^{3}) \right)^{3}
    \right).
\end{align*}
Noting $\|u_{(j)}\|_{\mathscr{H}^{s}(\mho)} \lesssim \e$ and $\|u\|_{\mathscr{H}^{s}(\mho)} \lesssim \e$, we have
\begin{align*}
	u  
	&=\sum_{j=1}^{3}\e_{(j)}u_{(j)} -   Q\l(h\l(\sum_{j=1}^{3}\e_{(j)}u_{(j)}\r)^{3}  \r) \mod |\e|^{5} \mathscr{B}_1^{s}(\mho)\\
	&=\sum_{j=1}^{3}\e_{(j)}u_{(j)} - \sum_{\alpha_1+\alpha_2+\alpha_3=3} \frac{3!}{\alpha_1 ! \alpha_{2}! \alpha_{3}!}   Q\l(h \prod_{j=1}^{3}\e_{(j)}^{\alpha_{j}} u_{(j)}^{\alpha_j}\r)    \mod |\e|^{5} \mathscr{B}_1^{s}(\mho),
\end{align*} within $\mho$.

For any  $1\le l \le 3$ and $1\le k_1< k_2<\cdots k_{l}\le 3$, we have
\begin{multline*}
	\lefteqn{    \P_{V,h}\l( \sum_{j=1}^l  \e_{(k_j)}f_{(k_j)} \r)  }\\ 
	= \sum_{j=1}^{l}\e_{(k_j)}u_{(k_j)} 
	-\sum_{\alpha_{k_1}+\cdots+\alpha_{k_l}=3}\frac{3!}{\alpha_{k_1} ! \cdots \alpha_{k_l}!} Q \l( h\prod_{j=1}^{l}\e_{(k_j)}^{\alpha_{k_j}} u_{(k_j)}^{\alpha_{k_j}}\r) 
	\mod |\e|^{5} \mathscr{B}_1^{s}(\mho)
\end{multline*} within $\mho$.

We only want to pick up the term including $u_{(123)} = Q\l(hu_{(1)}u_{(2)}u_{(3)}\r)$, which corresponds to the case $k_{\alpha_1} = k_{\alpha_2} = k_{\alpha_3} = 1$. 
To achieve this,  we recursively define 
\begin{align*}
	\notag  W_{V,h}(k) &:= \P_{V,h}\l( \e_{(k)}f_{(k)}\r)  \end{align*}for  $1\le k \le 3$;
\begin{align*}
	\notag  W_{V,h}(k_1,k_2) &:= \P_{V,h}\l( \e_{(k_1)}f_{(k_1)} + \e_{(k_2)}f_{(k_2)}\r) - W_{V,h}(k_1) - W_{V,h}(k_2)\end{align*}
for $1\le k_1<k_2\le 3$;
\begin{align*}
	\notag W_{V,h}(k_1,k_2,k_3) &:=  \P_{V,h}\l( \sum_{j=1}^{3}\e_{(k_j)}f_{(k_j)}\r) -\sum_{j=1}^{3}W_{V,h}(k_j)-\sum_{1\le i< j\le 3}W_{V,h}(k_i,k_j)\end{align*}
for $1\le k_1<k_2<k_3\le 3$.
It is easy to verify within $\mho$ that
\begin{align*}
	W_{V,h}(k) &=  \e_{(k)}u_{(k)}|_{\rotom} - \e_{(k)}^{3}Q(hu_{(k)}^{3})  \mod |\e|^{5} \mathscr{B}_1^{s}(\mho);
	\\
	W_{V,h}(k_1,k_2) &=   -\sum_{\alpha_1+\alpha_2=3,\alpha_j \ge 1}\frac{3!}{\alpha_{1}!\alpha_{2}!}    Q\l( h\prod_{j=1}^{2}\e_{(k_j)}^{\alpha_j}u_{(k_j)}^{\alpha_j}\r)  \mod |\e|^{5} \mathscr{B}_1^{s}(\mho);
	\\
	W_{V,h}(k_1,k_2,k_3) &=   -\sum_{\alpha_1+\alpha_2+\alpha_3=3,\alpha_j \ge 1}\frac{3!}{\alpha_{1}!\alpha_{2}!\alpha_3 !}    Q\l( h\prod_{j=1}^{3}\e_{(k_j)}^{\alpha_j}u_{(k_j)}^{\alpha_j}\r)    \mod |\e|^{5} \mathscr{B}_1^{s}(\mho). \end{align*} Consequently, we have \begin{align*}
	W_{V,h}(1,2,3) &= -3! \e_{(1)}\e_{(2)} \e_{(3)}u_{(123)}|_{\rotom}  \mod |\e|^{5} \mathscr{B}_1^{s}(\mho).
\end{align*}

Therefore, the leading error, within $\mho$, takes the form
\begin{align}\label{eqn: extraction of the term without repeated derivatives}
	\e_{(1)}\e_{(2)} \e_{(3)} \l(u_{(123)} -\tilde{u}_{(123)}\r)  
	= W_{V,h}(1,2,3)-W_{\tilde{V},\tilde{h}}(1,2,3)\mod |\e|^{5} \mathscr{B}_1^{s}(\mho).
\end{align}

It is easy to check that for 
$$
U =\{k_1,\dots, k_l : 1\le k_1 <\cdots <k_l\le 3, 1\le l \le 3\}
$$  there holds \begin{align} \label{eqn : error estimates for W}\| W_{V,h}(U) -W_{\tilde{V},\tilde{h}}(U) \|_{\mathscr{H}^{s}(\M_T)} \le C_l \delta.\end{align}
Then \eqref{eqn : error estimates for W} immediately yields
$$
\e_{(1)}\e_{(2)}\e_{(3)}\l\|u_{(123)}-\tilde{u}_{(123)} \r\|_{\mathscr{H}^{s}(\M_T)} \lesssim \delta + |\e|^{5} .
$$
We conclude the proof of Lemma \ref{lemma: stability estimate of the tri-linear operator} by taking $\e_{(1)} =  \e_{(2)} = \e_{(3)} = \delta^{{1}/{5}}$.
\end{proof}

\section{Stability of inversion }
\label{sec : stability estimates}

Section \ref{sec : trilinear operators} has reduced the quantitative information of  source-to-solution maps to trilinear measurement operators. As  the symbol calculus in Section \ref{sec : principal singularities}--\ref{sec : lower order symbol} suggests,  we shall derive the stabilities for $(V, h)$ through  the  symbols estimates encoded in trilinear operators.

\subsection{The nonlinearity}\label{Section: stability estimate of the nonlinear terms}  We first recover $h$, since the nonlinear coefficient appears in the principal symbol.


\begin{proof}[Proof of \eqref{eqn : stability h}]
	Recall the splitting in \eqref{decomposition of three-fold linearization}. Restricting to $\mho$, we have
	$$
	u_{(123)}|_{\mho} =  T_{V,h}\l(f_{(1)},f_{(2)},f_{(3)} \r)  =  v_{(123)}|_
	{\rotom}   + w_{(123)}|_{\rotom},
	$$
	where $v_{(123)}$ is a conormal distribution of order $3\mu+1/2$ over $\Lambda_1\cap T^{\ast}\rotom$ and  $w_{(123)}$ is smooth in a neighborhood of $\c{z}$ independent of parameter $r$. We remark that for a fixed small $r_0\in (0,1)$, $u_{(123)}$ is a continuous function in $r\in [0,r_{0}]$ and also in $(\c{x}_{(j)},\c{\xi}_{(j)},\c{y},\c{\eta}, \c{z},\c{\zeta})$. The variables $\c{\xi}_{(j)},\c{\eta},\c{\zeta}$ and $\c{x}_{(2)},\c{x}_{(3)}$ are fully determined by  $r$ and $\c{x} = \c{x}_{(1)},\c{y},\c{z}$ with
	\begin{equation}\label{eqn: domain of the continuous dependence points}
		\l(\c{x},\c{y},\c{z}\r) \in \mathbb{S}^{+}(\overline{\rotom}_1) = \l\{ \l(\c{x},\c{y},\c{z}\r) \in \mathbb{S}^{+}(\rotom) : \c{x},\c{z}\in \overline{\rotom}_1, \c{y}\in \overline{\D}_1\b\rotom  \r\}.
	\end{equation}  It suffices to prove uniform estimates for $h-\tilde{h}$ and $V-\tilde{V}$ with respect to $(\c{x},\c{y},\c{z},r)$. 
	
	On the other hand, \eqref{principal symbol of v_123} gives that 
	$$
	\sigma\l[ u_{(123)}- \tilde{u}_{(123)} \r](\c{z},\c{\zeta}) = r^{-6(\mu+1)}\l(h(\c{y}) - \tilde{h}(\c{y})\r)\A_{(123)}\l(\c{z},\c{\zeta};\c{y},\c{\eta};\c{X},\c{\Xi}\r).
	$$

	We shall invoke stabilities for $T_{V, h} - T_{\tilde{V}, \tilde{h}}$ in Lemma \ref{lemma: stability estimate of the tri-linear operator}  to estimate $h - \tilde{h}$.
	
	To compute the principal symbol, we find a half-density valued cut-off function $\phi$ supported in the vicinity of $\c{z}$ such that
	\begin{equation}\label{Choose of cut-off function on z}
		\phi(\c{z}) = 1, \quad \| \phi\|_{L^{1}(\M)}\le 1,
	\end{equation}
	and in addition select $\psi(x,\alpha)\in C^{\infty}(\M\times\Lambda_1)$ such that 
	\begin{equation}\label{eqn : cut-off on phase space for principal symbol}\left\{\begin{aligned}&\mbox{$\psi$ is homogeneous of degree $1$ in $\alpha$;} \\ &\mbox{the graph $x\mapsto d_x \psi(x,(\c{z},\c{\zeta}))$ intersects $\Lambda_1$ transversally at $(\c{z},\c{\zeta})$.}\end{aligned}\right.\end{equation}
	In particular,    $\psi$ is taken to be $\psi(z) = \langle z - \c{z},\c{\zeta}\rangle$ in local coordinates. 
	
	Invoking  the intrinsic formula \eqref{eqn: invariant definition of the principal symbol} and the homogeneity \eqref{homogeneity of interaction wave of three waves} of principal symbols, we have 
	\begin{multline}\label{stationary phase} 
		e^{\imath\psi(\c{z},(\c{z}, \c{\zeta}/\tau))}\langle u_{(123)}-\tilde{u}_{(123)}, \phi e^{-\imath\psi(\cdot,(\c{z},\c{\zeta}/\tau))}\rangle_{L^{2}(\M)}\\
		=  \tau^{-3\mu+\frac{1}{2}}\l( r^{-6(\mu+1)}\l(h(\c{y})- \tilde{h}(\c{y})\r)\omega^{-1}(\c{\zeta})\A_{(123)}+ g_{1}(\c{x},\c{y},\c{z},r,\tau)\r)
		,\quad \mbox{as $\tau\to 0+$}. 
	\end{multline} 
 The remainder estimates for stationary phase asymptotics, e.g. \cite[(3.5.8)]{Z}, guarantee that $g_1$ is uniformly bounded from above by
	\begin{equation}\label{uniform bound of the remiadner for the recovery of in cubic nonlinear case}
		|g_{1}| \lesssim\tau.
	\end{equation}

	We estimate the LHS of \eqref{stationary phase} as follows. 
	\begin{align}\label{estimate of the left-hand side of asymptotic expansion} 
		\lefteqn{\l| e^{\imath\psi(\c{z},(\c{z},\c{\zeta}/\tau))}\langle u_{(123)}-\tilde{u}_{(123)}, \phi e^{-\imath\psi(\cdot,(\c{z},\c{\zeta}/\tau))}\rangle_{L^{2}(\M)} \r| }
		\\
		\notag&\le \l\| u_{(123)}-\tilde{u}_{(123)}\r\|_{L^{\infty}(\rotom)} \l\| \phi\r\|_{L^{1}(\M)} \\
		\notag&\lesssim \l\| (T_{V, h}-T_{\tilde{V}, \tilde{h}})\l(f_{(1)},f_{(2)},f_{(3)} \r)\r\|_{\mathscr{H}^{s}\l(\rotom\r)} \lesssim \delta^{\frac{2}{5}}. 
	\end{align}

	Estimating the LHS of \eqref{stationary phase} by \eqref{estimate of the left-hand side of asymptotic expansion} and the remainder in the RHS of \eqref{stationary phase} by \eqref{uniform bound of the remiadner for the recovery of in cubic nonlinear case} respectively shows that 
	\begin{equation}\label{uniform bound of the principal symbols for the recovery of h}
		r^{-6(\mu+1)}\l|(h(\c{y})-\tilde{h}(\c{y}))\A_{(123)} \r| 
		\lesssim \tau +  \tau^{3\mu-\frac{1}{2}}\delta^{\frac{2}{5}}.
	\end{equation}

	Fix small $r = r_0 \in (0,1)$. Recall $\mu = -s-3/2$ and the uniform bounds of $\A_{(123)}$ in \eqref{uniform bound of A}. It follows that 
	$$
	\l|h(\c{y})-\tilde{h}(\c{y})\r| 
	\lesssim \tau
	+   \tau^{-3s-5}\delta^{\frac{2}{5}}.
	$$
	We conclude the proof by choosing  $\tau = \delta^{\frac{2}{15(s+2)}}$, i.e. $\tau
	=   \tau^{-3s-5}\delta^{2/5}$. 
\end{proof}

\subsection{The truncated integral}
Potential inversion \eqref{eqn : stability V} entails the following truncated integral transform of smooth functions.
\begin{definition}\label{definition: truncated intergral transformation}
	Let $ z\in \rotom, y\in \M$ and $\gamma(s)$ a parametrization of  a light-like geodesic  such that $\gamma(0)=y$, $\gamma(\tilde{s}) = z$ and  $|\dot{\gamma}_{\out}^{\flat}|=\sqrt{2}$ in local coordinates. The truncated integral transformation of $f\in C^{\infty}(\M)$ is defined as
	$$
	I_{f}(y,z) := \int_{0}^{\tilde{s}}f(\gamma(s))ds.
	$$
\end{definition}
In this section, we shall establish the following stability estimates.
\begin{proposition}\label{proposition: stability estimate of the truncated integral}
	For every $(y,z)\in \mathbb{L}$ with $y\in \overline{\D}_{2}\b\rotom$ and $z\in \overline{\rotom}_{2}$, there holds 
	\begin{equation}\label{eqn : stability for I_V}
		\sup_{y\in \overline{\D}_{2}\b\rotom,\ z\in \overline{\rotom_2}.} \l|I_{V-\tilde{V}}(y,z) \r|\lesssim  \l(\delta^{\frac{2}{5}} + \|h-\tilde{h} \|_{L^{\infty}(\overline{\D_1})}\r)^{\frac{2}{(3s+7)(6s+5)}} + \|h-\tilde{h} \|_{L^{\infty}(\overline{\D_1})} .
	\end{equation}
\end{proposition}

\begin{proof}

	First of all, we establish the estimates for $\|u_{(123)}^{\rem} - \tilde{u}_{(123)}^{\rem}\|$ in terms of $\|u_{(123)}-\tilde{u}_{(123)}\|$ and $\|u_{(123)}^{\fre}-\tilde{u}_{(123)}^{\fre}\|$. Observe that  \eqref{eqn: free wave for three-fold linearization} gives
	$$
	\Box \l(u_{(123)}^{\fre}-\tilde{u}_{(123)}^{\fre} \r) = (h-\tilde{h})u^{\fre}_{(1)}u^{\fre}_{(2)}u^{\fre}_{(3)}.
	$$ 
	Since $\D_1\subseteq \D$, the energy estimate \cite[pp. 404, Theorem 2]{Evans} shows 
	\begin{align*}
		\l\| u_{(123)}^{\fre} - \tilde{u}^{\fre}_{(123)}\r\|_{L^{\infty}\l(\rotom_{1}\r)} &\lesssim \l\| (h-\tilde{h})u_{(1)}^{\fre}u_{(2)}^{\fre}u_{(3)}^{\fre}\r\|_{L^2(\D_1)} \\
		&\lesssim \| h-\tilde{h} \|_{L^{\infty}(\overline{\D_1})}\prod_{j=1}^{3}\l\| u_{(j)}^{\fre} \r\|_{\mathscr{H}^{s}(\D)}  \\
		&\lesssim \| h-\tilde{h} \|_{L^{\infty}(\overline{\D_1})}.
	\end{align*}
	Combining this with Lemma \ref{lemma: stability estimate of the tri-linear operator}, we obtain 
	\begin{align}
		\label{stability estimate of the sub-leading term}   \lefteqn{\l\| u_{(123)}^{\rem} - \tilde{u}_{(123)}^{\rem}\r\|_{L^{\infty}(\rotom_1)}} \\
		\notag  &\le \l\| u_{(123)} - \tilde{u}_{(123)}\r\|_{L^{\infty}(\rotom_1)} +\l\| u_{(123)}^{\fre} - \tilde{u}_{(123)}^{\fre}\r\|_{L^{\infty}(\rotom_1)}\\
		\notag	&\lesssim \l\| \l(T_{V,h} - T_{\tilde{V},\tilde{h}}\r)\l(f_{(1)},f_{(2)},f_{(3)}\r)\r\|_{\mathscr{H}^{s}(\rotom_1)}  +\| h-\tilde{h} \|_{L^{\infty}(\overline{\D_1})}.\\
		\notag  & \lesssim \delta^{\frac{2}{5}}  +\| h-\tilde{h} \|_{L^{\infty}(\overline{\D_1})}.
	\end{align}

	Let $\phi, \psi$ be the cut-off functions as in \eqref{Choose of cut-off function on z} and \eqref{eqn : cut-off on phase space for principal symbol}. Then we locally have
	$$
	\sigma\l[u_{(123)}^{\rem}\r] = \sigma\l[v_{(123)}^{\rem}\r].
	$$ 
	In light of \eqref{principal symbol of the sub-leading term of three interaction waves}, $\sigma\l[u_{(123)}^{\rem}\r]$ is homogeneous of order $3\mu+1/2$. Thus, we have 
	\begin{multline}\label{stationary phase for subleading term}
		e^{\imath\psi(\c{z},(\c{z},\c{\zeta}/\tau))}\langle u_{(123)}^{\rem}-\tilde{u}^{\rem}_{(123)}, \phi e^{-\imath\psi(\cdot,(\c{z},\c{\zeta}/\tau))}\rangle_{L^{2}(\M)}  \\
		= \tau^{-3\mu+3/2}\l(\sigma\l[v_{(123)}^{\rem}-\tilde{v}_{(123)}^{\rem} \r](\c{z},\c{\zeta}) \omega^{-1}(\c{\zeta}) + g_{2}(\c{x},\c{y},\c{z},r,\tau)\r), \quad \text{as } \tau\to 0+,
	\end{multline}
	where $g_2$ obeys  $|g_2| \lesssim \tau$ uniformly.

	By \eqref{stability estimate of the sub-leading term}, the LHS of \eqref{stationary phase for subleading term} is bounded from above by
	$$
	\omega_{1} :=  \delta^{\frac{2}{5}} + \|h-\tilde{h} \|_{L^{\infty}(\overline{\D_1})}.
	$$   Analogous to the proof of \eqref{uniform bound of the principal symbols for the recovery of h}, the discrepancy of the principal symbols is uniformly bounded by 
	\begin{equation}\label{eqn: discrepancy of symbol}
		\l|\sigma\l[v_{(123)}^{\rem}-\tilde{v}_{(123)}^{\rem} \r](\c{z},\c{\zeta})\r| \lesssim \omega_{1} \tau^{3\mu-\frac{3}{2}}+ \tau.
	\end{equation}
	Choose $\tau = \omega_{1}^{1/(-3\mu+5/2)}$, i.e. $\omega_{1}\tau^{3\mu-3/2} = \tau$. Replacing the LHS of \eqref{eqn: discrepancy of symbol} by \eqref{principal symbol of the sub-leading term of three interaction waves} with $\mu = -s-3/2$ proves
	$$
	r^{3(2s+1)} \l|\mathcal{A}_{(123)}\r|\l|  h(\c{y})I_{V}(\c{y},\c{z}) -\tilde{h}(\c{y})I_{\tilde{V}}(\c{y},\c{z}) + O(r^2)\r| \lesssim \omega_{1}^{\frac{1}{3s+7}}.
	$$
	Since $|\A_{(123)}|^{-1}$ is uniformly bounded, we have
	$$
	\l|  h(\c{y})I_{V}(\c{y},\c{z}) -\tilde{h}(\c{y})I_{\tilde{V}}(\c{y},\c{z})\r|  \lesssim\omega_{1}^{\frac{1}{3s+7}}r^{-3(2s+1)}  + r^{2} .
	$$
 	Choose $r=\omega_{1}^{\frac{1}{(3s+7)(6s+5)}}$, i.e. $ \omega_{1}^{1/(3s+7)}r^{-3(2s+1)}   = r^{2}$.  It reduces to 
	$$
	\l|  h(\c{y})I_{V}(\c{y},\c{z}) -\tilde{h}(\c{y})I_{\tilde{V}}(\c{y},\c{z})\r|  \lesssim \omega_{1}^{\frac{2}{(3s+7)(6s+5)}}.
	$$
	Moreover, we rewrite the LHS as
	$$
	hI_{V} -\tilde{h}I_{\tilde{V}} = h I_{V-\tilde{V}} + (h-\tilde{h}) I_{\tilde{V}}.
	$$
	This, along with  \eqref{eqn : stability h} for $h - \tilde{h}$, shows 
	$$
	\l| I_{V-\tilde{V}} (\c{y},\c{z}) \r|  \lesssim     |h(\c{y})|^{-1} \l( \omega_{1}^{\frac{2}{(3s+7)(6s+5)}} + \|h-\tilde{h} \|_{L^{\infty}(\overline{\D_1})}\r).
	$$
	Since $h$ is non-vanishing, we conclude the proof of \eqref{eqn : stability for I_V} by taking the supreme of this inequality with respect to $\c{z}\in \overline{\rotom}_{2}$ and $\c{y}\in \overline{\D}_{2}\b\rotom$. 
\end{proof}

\subsection{The potential}\label{sec: mapping properties of truncated integrals}
Finally \eqref{eqn : stability V} is proved by \eqref{eqn : stability for I_V} and the following proposition.

\begin{proposition}\label{proposition: stability estimate of V}
	If for any $(y,z)\in \mathbb{L}$ with $z\in \overline{\rotom}_2$ and $y\in \overline{\D}_2$, 
	$$
	\l|I_{V-\tilde{V}}(y,z) \r| \lesssim \delta^{\mu_0},
	$$
	uniformly, then there holds 
	$$
	\| V-\tilde{V} \|_{L^{\infty}(\overline{\D}_2)} \lesssim \delta^{\frac{\mu_0}{2}}.
	$$
\end{proposition}

\begin{proof}

	Without loss of generality, for some $y \in \D_2 \setminus \mho$, we may set $V(y) - \tilde{V}(y) > 0$. Due to $V,\tilde{V}\in C^{\infty}(\M)$, there exists a sufficiently large $M$ such that 
	$$
	|\nabla V|\le M,\quad|\nabla \tilde{V}|\le M \quad\text{on } \D.
	$$
	We consider a function
	\begin{equation}\label{expression of the auxilary function}
		l(s) := V(y) - \tilde{V}(y) - 2\sqrt{2} M s, \quad \forall s \in [0,s_{\out}].
	\end{equation} and claim
	\begin{align}
		\label{eqn : property 1 of l} &V(\gamma_{\out}(s))-\tilde{V}(\gamma_{\out}(s))\ge l(s),\quad \forall s\in [0,s_{\out}]; \\
		\label{eqn : property 2 of l}&\mbox{ $\exists  s^{\ast}\in (0,s_{\out}]$ such that $l(s^{\ast})=0$. }
	\end{align}
	If \eqref{eqn : property 1 of l} is not true, then there exists $s_0\in [0,s_{\out}]$ such that
	\begin{equation}\label{assumption by contradiction}
		V(\gamma_{\out}(s_0))- \tilde{V}(\gamma_{\out}(s_{0}))< l(s_0).
	\end{equation}
	Since $V(y) - \tilde{V}(y) > 0$, we have $s_0 \neq 0$.  Then there exists $s_1\in (0, s_{0}]$ such that
	\begin{align*}
		\lefteqn{\l(V(\gamma_{\out}(s_0)) -\tilde{V}(\gamma_{\out}(s_0))\r) - \l(V(\gamma_{\out}(0)) -\tilde{V}(\gamma_{\out}(0))\r)}\\
		&=s_0 \times\l.\frac{d}{ds}\l(V(\gamma_{\out}(s)) -\tilde{V}(\gamma_{\out}(s))\r)\r|_{s=s_1}  \ge -2\sqrt{2}M s_0.
	\end{align*} 
	However,  \eqref{assumption by contradiction} yields that
	\begin{align*}
		\lefteqn{\l(V(\gamma_{\out}(s_0)) -\tilde{V}(\gamma_{\out}(s_0))\r) - \l(V(\gamma_{\out}(0)) -\tilde{V}(\gamma_{\out}(0))\r)}\\
		&<l(s_0) - V(\gamma_{\out}(0)) +\tilde{V}(\gamma_{\out}(0)) \\
		&=V(\gamma_{\out}(0)) -\tilde{V}(\gamma_{\out}(0)) - 2\sqrt{2}Ms_0 - V(\gamma_{\out}(0)) +\tilde{V}(\gamma_{\out}(0))\\
		&= -2\sqrt{2}M s_0.
	\end{align*}
	Moreover,  \eqref{eqn : property 2 of l}  readily follows from $(V-\tilde{V})|_{\rotom} \equiv 0$.

	Combining \eqref{eqn : property 1 of l} and \eqref{eqn : property 2 of l} together shows that 
	\begin{multline*} \int_{0}^{s_{\out}} V(\gamma_{\out}(s))-\tilde{V}(\gamma_{\out}(s))ds- \int_{s^{\ast}}^{s_{\out}}  V(\gamma_{\out}(s))-\tilde{V}(\gamma_{\out}(s))ds
		\\  = \int_{0}^{s^{\ast}}V(\gamma_{\out}(s))- \tilde{V}(\gamma_{\out}(s))ds \ge\int_{0}^{s^{\ast}} l(s) ds >0 
	\end{multline*}

	Write $\gamma_{\out}(s^{\ast}) := y^{\ast}$. If $y^{\ast}\in \rotom$, the second term on the LHS vanishes. If $y^{\ast}\notin \rotom$, The LHS is equal to
	$$
	I_{V-\tilde{V}}(y,z) - I_{V-\tilde{V}}(y^{\ast},z),
	$$
	and bounded from above by $C\delta^{\mu_0 }$. Thus,
	$$
	0< \int_{0}^{s^{\ast}} l(s) ds \le C\delta^{\mu_0}.
	$$
	Moreover, \eqref{eqn : property 2 of l} implies 
	$$
	s^{\ast} =\frac{V-\tilde{V}}{2\sqrt{2}M}.
	$$
	Integrating $l(s)$ from $0$ to $s^{\ast}$ yields that
	$$
	\frac{\l|V(y)-\tilde{V}(y)\r|^{2}}{4\sqrt{2}M}\le C\delta^{\mu_0},
	$$
	which proves the proposition.
\end{proof}

	\bigskip

\noindent {\bf Acknowledgements.} The authors were supported in part by Natural Science Foundation of Shanghai 23JC1400501. Part of the work was done during X.C.’s visit to Research Center for Mathematics of Data (MoD) at
Friedrich-Alexander-Universit\"{a}t Erlangen-N\"{u}rnberg under the support of  NSFC-DFG Sino-German Mobility Programme M-0548.

\bigskip	\noindent {\bf Data Availability Statement.} Data sharing not applicable to this article as no datasets were generated or analysed during the current study.

\bigskip	\noindent {\bf Conflict of Interest.} The authors have no conflicts of interest to declare that are relevant to the content of this article.

\bibliographystyle{abbrv}
\bibliography{reference}

\end{document}